\documentclass[12pt,amssymb,verbatim]{amsart}

\usepackage{amsfonts,latexsym,geometry,color,hyperref}
\usepackage{amssymb,graphics,graphicx}
\usepackage{amsfonts,latexsym,geometry}
\usepackage{amssymb,graphics,graphicx}
\usepackage{amscd}
\usepackage{amssymb}
\usepackage{amsthm}
\usepackage[utf8]{inputenc} 
\usepackage[T1]{fontenc}
\usepackage{amsmath}
\usepackage{amsfonts}
\usepackage{amssymb}
\usepackage{graphicx} 
\usepackage{enumerate}
\newcommand\supp{\mathrm{supp}}
\newtheorem{theoreme}{Theorem}[section] %
\newtheorem{proposition}[theoreme]{Proposition} %
\newtheorem{corollary}[theoreme]{Corollary} %
\newtheorem{lemme}[theoreme]{Lemma} %
\newtheorem{definition}{Definition}[section] %
\newtheorem{remark}[theoreme]{Remark} %
\newtheorem{remarque}[theoreme]{Remark} %

\newcommand\mk{\medskip}
\newcommand\sk{\smallskip}
\newcommand\N {\mathbb{N}}
\newcommand\R{\mathbb{R}} 
\newcommand\zu{[0,1]} 
\newcommand\dimm{\underline{\dim}_H}

\RequirePackage{yhmath} 
\renewcommand\widering[1]{\ring{#1}}

\newcommand\ep{\varepsilon}
 
\begin{document}
\title{An heterogeneous ubiquity theorem, application to self-similar measures}
\maketitle
\author{Edouard Daviaud,
 Université Paris-Est, LAMA (UMR 8050) UPEMLV, UPEC, CNRS, F-94010, Créteil, France}
\section{Introduction}
Estimating the Hausdorff dimension of limsup sets obtained from the contractions of the elements of a given family of sets is a natural question of metric approximation theory, which  arises in many contexts. In this article, given a sequence  of balls of $\mathbb{R}^d$, $\mathcal{B}= (B(x_n,r_n))_{n\in\mathbb{N}}$, we investigate   in a very general frame the size properties of the limsup sets obtained from smaller  sets, i.e sets of the form $\limsup_{n\to +\infty} U_n $, where $U_n \subset B_n$.


Let us recall that  the historical example of Jarnik-Besicovitch's theorem deals with the case    $U_n=B_n^\delta:= B(x_n,r_n^{\delta })$, where $\delta>1$, $x_n$ is a rational number $\frac{p}{q}$ and $r_n =\frac{1}{q^2}$. More generally in metric number theory, one often aims at computing the Hausdorff dimension of sets $\limsup_{n\to\infty} U_n$, where $(U_n)_{n\in\mathbb{N}}$ has some algebraic or dynamical meaning.  
Generalizations of Jarnik-Besicovitch's Theorem often  consider a given  sequence $(x_n)_{n\in\mathbb{N}}$ of points in  $\R^d$, as well as a sequence of radii $(r_n)_{n\in\mathbb{N}}$  for which the associated limsup  set  $E_{1}=\limsup_{n\rightarrow+\infty}B_n $ has a controlled size (in terms of Lebesgue measure or Hausdorff dimension for instance);  then, given a sequence of sets $\mathcal{U}=(U_n)_{n\in\mathbb{N}}$ with, for every $n\in\mathbb{N}$, $U_n \subset B_n$, one estimates the Hausdorff dimension of the smaller limsup set  $E(\mathcal{U})=\limsup_{n\rightarrow+\infty}U_n$. The classical case is when the set $U_n$ is a shrunk ball $B_n^{\delta}$, for some $\delta>1$, that is $E(\mathcal{U})$ is  the limsup set of the $\delta$-contracted balls, but different shapes for $U_n$ have also been considered (rectangles or ellipsoïds rather than balls for instance). Such problems are studied for instance in \cite{Ja,BV,BS2,RE,KR} among many references. 

The same question arises on any topological dynamical system $(X,T)$ endowed with some metric, when the  sequence $(x_n)_{n\in\mathbb{N}}$ is the orbit $(T^{n}(x))_{n\in\mathbb N}$ of a well chosen point $x$. Some specific cases are for instance treated in  \cite{HV,LS,PR}. In probability theory, the famous Dvoretzky covering problem consists in computing, when it is possible, the Hausdorff dimension of the limsup set associated with a   sequence of random balls drawn independently and  uniformly in a compact Baire space, see for instance  \cite{FJS,EP,BS2}. {In analysis, the value of the pointwise regularity exponents of measures and functions at a given point $x$ often relies on the ability to understand how   $x$ is close to remarkable points $x_n$. The reader may refer to \cite{Ja,BS3,BS4}.

{As mentioned above, in the largest part of the literature, a strong     geometric measure theoretic condition  is initially  imposed on $\mathcal{B}$ to obtain results, for instance that the Lebesgue measure of $\limsup_{n\to +\infty} B(x_n,r_{  n})$ is full (cf \cite{BV}). But there are many situations in which the Lebesgue measure is not the relevant measure to work with (cf \cite{BS}).}

{Our purpose in this article is to obtain a general lower bound for the Hausdorff dimension $\limsup_{n\to +\infty} U_n$, where the sets $(U_n)$ are open sets in some balls $(B_n)$ satisfying the property  called $\mu$-asymptotically covering property, where $\mu$ is a probability measure on $\R^d$. This property, introduced in \cite{ED2},   is   proved to be almost equivalent to verifying  that $\mu(\limsup_{n\rightarrow+\infty}B_n)=1$ (e.g \cite{ED2}). 

The results presented here extend, for instance,  both the results of Koivusalo-Rams stated in \cite{KR} and the result of Barral-Seuret (\cite{BS2}) which deals with balls and self-similar measures under the open set condition. It is worth noticing that the work of Koivusalo and Rams in \cite{KR} highlighted the importance of the Hausdorff content to compute Hausdorff dimension of limsup sets and this article makes further use of this fact. 

An important advantage  of the lower bound obtained in the present paper is that its value is tractable in many cases. For instance, as a first application, a ubiquity theorem is given in the case where $\mu$ is a   self-similar measure (we do not require any condition on the possible overlaps  associated with such a $\mu$). 

Two other applications of our main result are treated in this article as well. The problem of self-similar shrinking targets is studied when the corresponding iterated function system (in short IFS) is dimension-regular and has similarity dimension less than $d$, meaning in particular that for every self-similar measure, the similarity dimension and the Hausdorff dimension coincide (see Section \ref{shrtag}, Definition \ref{dimreg}).

Another application in Diophantine approximation is given. Let $K_{1/3}^{(0)}$ the set of points of $[0,1]$ such that in their sequence of digits in basis $3$, the asymptotic frequency of appearance of the digit $1$ is  infinitely many often close to $0$ (note that this set contains the middle-third Cantor set $K_{ 1/3}$ and $\dim_H (K_{1/3}^{(0)})=\dim_H (K_{1/3})$). We compute the Hausdorff dimension of points of $K_{1/3}^{(0)}$ well approximable by rational (see Theorem \ref{main} for a precise statement).

\section{Definitions and main statements}\label{sec1}

Let us start with some notations 

 Let $d$ $\in\mathbb{N}$. For $x\in\mathbb{R}^{d}$, $r>0$,  $B(x,r)$ stands for the closed ball of ($\mathbb{R}^{d}$,$\vert\vert \ \ \vert\vert_{\infty}$) of center $x$ and radius $r$. 
 Given a ball $B$, $\vert B\vert$ stands for the diameter of $B$. For $t\geq 0$, $\delta\in\mathbb{R}$ and $B=B(x,r)$,   $t B$ stand for $B(x,t r)$, i.e. the ball with same center as $B$ and radius multiplied by $t$,   and the  $\delta$-contracted  ball $B^{\delta}$ is  defined by $B^{\delta}=B(x ,r^{\delta})$.
\smallskip

Given a set $E\subset \mathbb{R}^d$, $\widering{E}$ stands for the  interior of the $E$, $\overline{E}$ its  closure and $\partial E$ its boundary, i.e, $\partial E =\overline{E}\setminus \widering{E}$. If $E$ is a Borel subset of $\R^d$, its Borel $\sigma$-algebra is denoted by $\mathcal B(E)$.
\smallskip

Given a topological space $X$, the Borel $\sigma$-algebra of $X$ is denoted $\mathcal{B}(X)$ and the space of probability measure on $\mathcal{B}(X)$ is denoted $\mathcal{M}(X).$ 

\sk
 The $d$-dimensional Lebesgue measure on $(\mathbb R^d,\mathcal{B}(\mathbb{R}^d))$ is denoted by 
$\mathcal{L}^d$.
\smallskip

For $\mu \in\mathcal{M}(\R^d)$,   $\supp(\mu)=\left\{x\in[0,1]: \ \forall r>0, \ \mu(B(x,r))>0\right\}$ is the topological support of $\mu$.
\smallskip

 Given $E\subset \mathbb{R}^d$, $\dim_{H}(E)$ and $\dim_{P}(E)$ denote respectively  the Hausdorff   and the packing dimension of $E$.
\smallskip
 
Now we recall some definitions.


\begin{definition}
\label{hausgau}
Let $\zeta :\mathbb{R}^{+}\mapsto\mathbb{R}^+$. Suppose that $\zeta$ is increasing in a neighborhood of $0$ and $\zeta (0)=0$. The  Hausdorff outer measure at scale $t\in(0,+\infty]$ associated with $\zeta$ of a set $E$ is defined by 
\begin{equation}
\label{gaug}
\mathcal{H}^{\zeta}_t (E)=\inf \left\{\sum_{n\in\mathbb{N}}\zeta (\vert B_n\vert) : \, \vert B_n \vert \leq t, \ B_n \text{ closed ball and } E\subset \bigcup_{n\in \mathbb{N}}B_n\right\}.
\end{equation}
The Hausdorff measure associated with $\zeta$ of a set $E$ is defined by 
\begin{equation}
\mathcal{H}^{\zeta} (E)=\lim_{t\to 0^+}\mathcal{H}^{\zeta}_t (E).
\end{equation}
\end{definition}

For $t\in (0,+\infty]$, $s\geq 0$ and $\zeta:x\mapsto x^s$, one simply uses the usual notation $\mathcal{H}^{\zeta}_t (E)=\mathcal{H}^{s}_t (E)$ and $\mathcal{H}^{\zeta} (E)=\mathcal{H}^{s} (E)$, and these measures are called $s$-dimensional Hausdorff outer measure at scale $t\in(0,+\infty]$ and  $s$-dimensional Hausdorff measure respectively. Thus, 
\begin{equation}
\label{hcont}
\mathcal{H}^{s}_{t}(E)=\inf \left\{\sum_{n\in\mathbb{N}}\vert B_n\vert^s : \, \vert B_n \vert \leq t, \ B_n \text{  closed ball and } E\subset \bigcup_{n\in \mathbb{N}}B_n\right\}. 
\end{equation}
The quantity $\mathcal{H}^{s}_{\infty}(E)$ (obtained for $t=+\infty$) is called the $s$-dimensional Hausdorff content of the set $E$.

\begin{definition} 
\label{dim}
Let $\mu\in\mathcal{M}(\mathbb{R}^d)$.  
For $x\in \supp(\mu)$, the lower and upper  local dimensions of $\mu$ at $x$ are  defined as
\begin{align*}
\underline\dim_{{\rm loc}}(\mu,x)=\liminf_{r\rightarrow 0^{+}}\frac{\log(\mu(B(x,r)))}{\log(r)}
  \ \mbox{ and } \    \overline\dim_{{\rm loc}}(\mu,x)=\limsup_{r\rightarrow 0^{+}}\frac{\log (\mu(B(x,r)))}{\log(r)}.
 \end{align*}
Then, the lower and upper Hausdorff dimensions of $\mu$  are respectively defined by 
\begin{equation}
\label{dimmu}
\dimm(\mu)={\mathrm{ess\,inf}}_{\mu}(\underline\dim_{{\rm loc}}(\mu,x))  \ \ \mbox{ and } \ \ \overline{\dim}_P (\mu)={\mathrm{ess\,sup}}_{\mu}(\overline\dim_{{\rm loc}}(\mu,x)).
\end{equation}

\end{definition}

It is known (for more details see \cite{F}) that
\begin{equation*}
\begin{split}
\dimm(\mu)&=\inf\{\dim_{H}(E):\, E\in\mathcal{B}(\mathbb{R}^d),\, \mu(E)>0\} \\
\overline{\dim}_P (\mu)&=\inf\{\dim_P(E):\, E\in\mathcal{B}(\mathbb{R}^d),\, \mu(E)=1\}.
\end{split}
\end{equation*}
When $\underline \dim_H(\mu)=\overline \dim_P(\mu)$, this common value is simply denoted by $\dim(\mu)$ and~$\mu$ is said to be \textit{exact dimensional}.

\medskip

\subsection{The $\mu$-a.c property}

\mk

We fix a sequence of closed balls $\mathcal B=(B_n)_{n\in\mathbb N}$ such that    $\lim_{n\to +\infty} |B_n| = 0$ (otherwise the situation is trivial for the questions we consider).

\mk

The main  property (introduced in \cite{ED2}) used for the sequence of balls $\mathcal{B}$  is meant to ensure  that any set can be covered   efficiently  by the limsup of the $B_n$'s, with respect to a measure $\mu$.  This property  is  a general version of the key covering property used in the KGB Lemma of  Beresnevitch and Velani, stated in \cite{BV}, using a Borel probability measure~$\mu$. Observe that such properties (like the KGB Lemma) are   usually   key   (cf \cite{Ja,BV,BS} for instance) to prove ubiquity or mass transference results.

 \begin{definition} 
\label{ac}
Let   $\mu\in \mathcal{M}(\R^d)$. The  sequence $\mathcal{B}= (B_n)_{n\in\mathbb{N}}$ of balls of $\R^d$   is said to be $\mu$-asymptotically covering (in short,  $\mu$-a.c) when  there exists  a constant $C>0$ such that for every open set $\Omega\subset \R^d $ and $g\in\mathbb{N}$, there is an integer  $N_\Omega \in\mathbb{N}$ as well  as $g\leq n_1 \leq ...\leq n_{N_\Omega}$ such that: 
\begin{itemize}
\item [(i)]$\forall \, 1\leq i\leq N_\Omega$, $B_{n_i}\subset \Omega$;
\item [(ii)]$\forall \, 1\leq i\neq j\leq N_\Omega$, $B_{n_i}\cap B_{n_j}=\emptyset$;
\item  [(iii)] also,
\begin{equation}
\label{majac}
\mu\Big(\bigcup_{i=1}^{N_\Omega}B_{n_i} \Big )\geq C\mu(\Omega).
\end{equation}
\end{itemize}
\end{definition}

In other words, for any open set $\Omega$ and any integer $g\ge 1$, there exits a finite set of disjoint balls of $\left\{B_n\right\}_{n\geq g}$ supporting a fixed proportion of $\mu(\Omega)$.

This notion of $\mu$-asymptotically covering  is related to the way the balls of $\mathcal{B}$ are distributed according to the measure $\mu$. This property  is a priori slightly stronger  than having a  $\limsup$ of full $\mu$-measure when $\mu$ is not doubling,  as suggested by the following lemma proved in \cite{ED2}, and whose second item will be used to apply our main theorem to self-similar measures. However, it follows from the proof of  \cite[Lemma 5]{BV} that these properties are equivalent when $\mu$ is doubling.

\begin{lemme} 
\label{equiac}
Let   $\mu\in\mathcal{M}(\R^d)$  and $\mathcal{B} =(B_n :=B(x_n ,r_n))_{n\in\mathbb{N}}$ be a sequence of balls of  $\R^d$ with $\lim_{n\to +\infty} r_{n}= 0$.
\begin{enumerate}
\smallskip
\item
If $\mathcal{B} $ is $\mu$-a.c, then $\mu(\limsup_{n\rightarrow+\infty}B_n)=1.$
\smallskip
\item
 If there exists $v<1$ such that $ \mu \big(\limsup_{n\rightarrow+\infty}(v B_n) \big)=1$, then $\mathcal{B} $ is $\mu$-a.c.
\end{enumerate}
\end{lemme}

\subsection{Essential content and statement of the main result}

\mk

The key geometric notion for the ubiquity theorem developed in this paper is the following. 
\begin{definition} 
\label{mucont}
{Let $\mu \in\mathcal{M}(\R^d)$, and $s\geq 0$.
The $s$-dimensional $\mu$-essential Hausdorff content at scale $t\in(0,+\infty]$ of a set $A\subset \mathcal B(\R^d)$ is defined as}
{\begin{equation}
\label{eqmucont}
 \mathcal{H}^{\mu,s}_{t}(A)=\inf\left\{\mathcal{H}^{s}_{t}(E): \ E\subset  A , \ \mu(E)=\mu(A)\right\}.
 \end{equation}}
\end{definition}
One will almost exclusively look at these contents at scale $t=+\infty$ and one refers to   $ \mathcal{H}^{\mu,s}_{\infty}(A)$  as the {\em $s$-dimensional  $\mu$-essential Hausdorff content of $A$}. Basic properties of those quantities are studied in Section \ref{sec-mino}, and precise estimates  of $ \mathcal{H}^{\mu,s}_{\infty}(A)$ are achieved for the Lebesgue measure and self-similar measures  in Section \ref{sec-example}. 

Note that in \cite[Theorem 3.1]{KR}  the key underlying geometric notion used to handle the variety of shapes of the sets $(U_n)_{n\in\mathbb{N}}$ is the Hausdorff content. It is easily seen from  \eqref{hcont} that the Hausdorff content also carries some ``high scale'' geometric information (because there is no restriction concerning the diameter of the balls $(B_n)$ in \eqref{hcont}). This will also be the case in this article to handle not only the shape of the sets $(U_n)_{n\in\mathbb{N}}$ but also the geometric behavior related to the measure $\mu$ at high scale in the sets $(U_n)_{n\in\mathbb{N}}.$ 


\medskip

The $s$-dimensional  $\mu$-essential Hausdorff content is now used to associate a critical exponent to any sequence of open sets $(U_n)_{n\in\mathbb N}$ such that $U_n\subset B_n$ for all $n\in\mathbb N$. This exponent is involved in our lower bound estimate of $\dim_H (\limsup_{n\to +\infty} U_n)$. 

\begin{definition} 
\label{expoani}
Let $\mu\in\mathcal{M}(\R^d)$. If $B$ and  $U$ are Borel subsets of $\R^d$,  the {\em $\mu$-critical exponent} of  $(B,U)$ is defined as
\begin{equation}
\label{defsndani}
 s_\mu(B,U) =\sup\left\{s\geq 0 \ : \ \mathcal{H}^{\mu,s}_{\infty}\left (U\right)\geq \mu( B)\right\}.
  \end{equation}

Let  $\mathcal{B}= (B_n)_{n\in\mathbb{N}}$ be a sequence of closed balls, $\mathcal{U}= (U_n)_{n\in\mathbb{N}}$ a sequence of Borel subsets of $\R^d$, and $s\geq 0$. 

Let
 \begin{align}
 \label{defNsdani}
 \mathcal{N}_\mu(\mathcal{B},\mathcal{U},s) & =\left\{n\in\mathbb{N} \ : s_\mu(B_n,U_n) \geq s\right\}. 
 \end{align}
Then, define the $\mu$-critical exponent of $(\mathcal{B},\mathcal{U})$ as
 \begin{align}
\label{defsmdani}
 s(\mu,\mathcal{B},\mathcal{U}) & =\sup\left\{s \geq 0 : \ (B_n)_{n\in\mathcal{N}_\mu(\mathcal{B},\mathcal{U},s)}  \mbox{ is } \mu\mbox{-a.c.} \right\}.
  \end{align}
\end{definition}

It is worth noting that, for $s^{\prime}\leq s$, one has $  \mathcal{N}_\mu(\mathcal{B},\mathcal{U},s)\subset \mathcal{N}_\mu(\mathcal{B},\mathcal{U},s^{\prime}) $.

\medskip 

The main result of this paper is the following.
  
  \begin{theoreme} 
\label{lowani}
Let  $\mathcal{B} =(B_n)_{n\in \N} $ be a   sequence of closed balls of $\R^d$ such that $\vert B_n \vert \to 0$ and $\mathcal{U}=(U_n)_{n\in\mathbb{N}}$ a sequence of open sets such that $U_n \subset B_n$ for all $n\in\mathbb{N}$.

Then, for every   $\mu \in\mathcal{M}(\R^d)$ such that $\min\left\{s(\mu , \mathcal{B},\mathcal{U}),\dimm(\mu)\right\}>0$ there exists a gauge function $\zeta:\R^+\to \R^+$ such that $\lim_{r\to 0^+} \frac{\log \zeta(r)}{\log r} =\min\left\{ s(\mu, \mathcal{B},\mathcal{U}),\dimm(\mu)\right\}$  and  
$$\mathcal{H}^\zeta( \limsup_{n\rightarrow +\infty}U_n) >0.$$
In particular, for every   $\mu \in\mathcal{M}(\R^d)$, one has
 \begin{equation}
\label{conc2ani}
 \dim_{H}\left(\limsup_{n\rightarrow +\infty}U_n\right)\geq \min\left\{s(\mu , \mathcal{B},\mathcal{U}),\dimm(\mu)\right\}.
 \end{equation}
\end{theoreme}

\begin{remarque}
(1) It is easily verified that the lower-bound in Theorem \ref{lowani} equals  $-\infty$ if the sequence $(B_n)_{n\in\mathbb{N}}$ is not assumed to be  $\mu$-a.c. Consequently, for the previous result to give non trivial information one has to assume that $(B_n)_{n\in\mathbb{N}}$ is  $\mu$-a.c. The question is then to give more explicit estimates  of $s(\mu , \mathcal{B},\mathcal{U})$ depending on the specifities of $(\mu , \mathcal{B},\mathcal{U})$. 
\mk


(2) It is proved in Section~\ref{sec-mino} that $s(\mu , \mathcal{B},\mathcal{U})\leq \overline{\dim}_H (\mu).$ This implies that for exact dimensional measures, $\min\left\{s(\mu , \mathcal{B},\mathcal{U}),\dimm(\mu)\right\}=s(\mu , \mathcal{B},\mathcal{U}).$\mk

(3) The case where $\mu$ satisfies $\min\left\{s(\mu , \mathcal{B},\mathcal{U}),\dimm(\mu)\right\}=0$ could also be treated, but  although \eqref{conc2ani} is still obviously true, some distinction should further be made when investigating the existence of the gauge function. If $\mathcal{H}^{\mu,s}_{\infty}(U_n)=0$ for any $n\in\mathbb{N}$,  the set $\limsup_{n\rightarrow+\infty} U_n$ could, for instance, be empty. On the other hand, if $(B_n)_{n\in\mathbb{N}}$ is $\mu$-a.c and $s_{\mu}(B_n ,U_n)>0$ for any $n\in\mathbb{N}$, a gauge function can be constructed in a similar way than in the proof of Theorem \ref{lowani}. However that the existence of such a gauge function in the case  $\min\left\{s(\mu , \mathcal{B},\mathcal{U}),\dimm(\mu)\right\}=0$, is of little interest for practical applications, and is not treated in this article.
\end{remarque}


A quite direct, but useful, corollary of Theorem \ref{lowani} is the following: 
\begin{corollary} 
\label{zzani}
Let $\mu\in\mathcal{M}(\R^d)$ and $\mathcal{B}= (B_n)_{n\in\mathbb{N}}$ be a   $\mu$-a.c. sequence of  closed balls of~$\R^d$.
Let $\mathcal{U}=(U_n)_{n\in\mathbb{N}}$ be a sequence of open sets such that 
$U_n \subset B_n$ for all $n\in \mathbb N$, and $0\leq s\leq \dimm(\mu)$. If $ \limsup_{n\rightarrow+\infty}
\frac{\log\mathcal{H}^{\mu,s}_{\infty}(U_n )}{\log \mu( B_n )} \leq 1$, then 
$s(\mu,\mathcal B,\mathcal U)\ge s$, so that 
$$\dim_{H}(\limsup_{n\rightarrow +\infty}U_n)\geq s .
$$

\end{corollary}

In the classical case where the sets $U_n$ are shrunk balls of the form $B_n ^{\delta}$ (with $\delta\geq 1$), it is convenient to consider the following quantity: 

\begin{definition}
\label{deftdelta}
Let $\mu \in\mathcal{M}(\R^d)$, $\varepsilon>0$ and $\mathcal{B}= (B_n)_{n\in\mathbb{N}}$ be a sequence of balls of $\R^d.$ For every $\delta \ge 1$, set
\begin{equation}
\label{tdeltainf}
t(\mu, \delta,\varepsilon ,\mathcal{B} )=\limsup_{n\rightarrow+\infty}\frac{\log(\mathcal{H}_{\infty}^{\mu,\dimm(\mu)-\varepsilon}(\widering{B}_n ^{\delta}))}{\log(\vert B_n ^{\delta} \vert)}.
\end{equation}
Then the $(\mu,\delta)$-exponent of the sequence $\mathcal{B} $ is defined as
\begin{equation}
\label{muexpo}
t(\mu,\delta,\mathcal{B} )=\lim_{\varepsilon\to 0}t(\mu,\delta,\varepsilon ,\mathcal{B}).
\end{equation} 

\end{definition}

It follows from the definitions that $t(\mu,\delta,\mathcal{B} )$ exists as a limit, since $\ep\mapsto t(\mu, \delta,\varepsilon ,\mathcal{B} )$ is monotonic. Moreover, one has $\dimm(\mu)\le t(\mu,\delta,\mathcal{B})$ (see the proof of Corollary~\ref{minoeffec}).

Next result provides a more explicit lower bound estimate of the Hausdorff dimension of the limsup  of $\delta$-contracted balls; it is a consequence of Corollary~\ref{zzani}.

\begin{corollary}
\label{minoeffec}
Let $\mu \in\mathcal{M}(\R^d)$ and  $\mathcal{B}= (B_n)_{n\in\mathbb{N}}$  a $\mu$-a.c  sequence of closed balls of~$\R^d$. 
Suppose that $\dimm(\mu)>0$. For every $\delta\geq 1$, setting 

$$
s_\delta= \frac{\dimm(\mu)}{\delta} \cdot \frac{\dimm(\mu)} {t(\mu,\delta,\mathcal{B})},
$$
one has 
$$
s\big (\mu, (B_n)_{n\in\N},(\widering B_n ^{\delta})_{n\in\N}\big )\ge s_\delta,
$$ 
hence
$$\dim_H (\limsup_{n\rightarrow+\infty}B_n ^{\delta})\ge \dim_H (\limsup_{n\rightarrow+\infty}\widering B_n ^{\delta})\ge s_\delta.$$

\end{corollary}

\subsection{Application to self-similar measures}

Let us start by recalling the definition of a self-similar measure.

%

\begin{definition}
\label{def-ssmu}  A self-similar IFS is a family $S=\left\{f_i\right\}_{i=1}^m$ of $m\ge 2$ contracting similarities  of $\mathbb{R}^d$. 

Let $(p_i)_{i=1,...,m}\in (0,1)^m$ be a positive probability vector, i.e. $p_1+\cdots +p_m=1$.

The self-similar measure $\mu$ associated with $ \left\{f_i\right\}_{i=1}^m$  and $(p_i)_{i=1}^m$ is the unique probability measure such that 
\begin{equation}
\label{def-ssmu2}
\mu=\sum_{i=1}^m p_i \mu \circ f_i^{-1}.
\end{equation}

The topological support of $\mu$ is the attractor of $S$, that is the unique non-empty compact set $K\subset X$ such that  $K=\bigcup_{i=1}^m f_i(K)$.

\end{definition}

The existence  and uniqueness of $K$ and $\mu$ are standard results \cite{Hutchinson}. Recall that due to a result by Feng and Hu \cite{FH} any self-similar measure is exact dimensional. 

The essential Hausdorff contents of a self-similar measure  $\mu$ can be estimated quite precisely. 
\begin{theoreme}
\label{contss}
Let $S$ be a self-similar IFS of $\R^d$. Let $K$ be the attractor of $S$. Let~$\mu$ be a self-similar measure  associated with $S$. For any $0\leq s<\dim(\mu)$, there exists a constant $c=c(d,\mu,s)>0 $ depending on the dimension $d$, $\mu$ and  $s$  only, such that for any ball $B=B(x,r)$ centered on $K$ and $r\leq 1$, any open set $\Omega$, one has 
\begin{align}
\label{genhcontss}
&c(d,\mu,s)\vert B\vert ^{s}\leq\mathcal{H}^{\mu, s }_{\infty}(\widering{B})\leq  \mathcal{H}^{ \mu, s}_{\infty}(B)\leq\vert B\vert ^{s}\text{ and } \nonumber\\
&c(d,\mu,s)\mathcal{H}^{s}_{\infty}(\Omega \cap K)\leq \mathcal{H}^{\mu,s}_{\infty}(\Omega)\leq \mathcal{H}^{s}_{\infty}(\Omega \cap K).
\end{align}
For any $s>\dim(\mu)$, $ \mathcal{H}^{\mu,s}_{\infty}(\Omega)=0.$

\end{theoreme}

\begin{remarque}
(1) The system $S$ is not assumed to verify any separation condition.\mk

(2) In the special case of the Lebesgue measure restricted to  $[0,1]^d$ (or some cube $K$), \eqref{genhcontss} implies that  for any $0 \leq s\leq d$, the Lebesgue-essential $s$ dimensional Hausdorff content is strongly equivalent to the usual $s$-dimensional Hausdorff content (it is even possible to take the constant $c(d,\mu,s)$ in \eqref{genhcontss}, independent of~$s$), so that Theorem \ref{contss} together with Theorem \ref{lowani} in the this special case implies the main theorem of Koivusalo and Rams, \cite[Theorem 3.2]{KR}, recalled below.
\end{remarque}

\begin{theoreme}[\cite{KR}]
\label{lowkr}
Let  $(B_n)_{n\rightarrow+\infty}$ be a sequence of balls of $[0,1]^d$ verifying $\vert B_n \vert \to 0$ and $\mathcal{L}^d (\limsup_{n\rightarrow+\infty}B_n)=1$. 

Let $(U_n)_{n\in \mathbb{N}}$ be a sequence of open sets satisfying $U_n \subset B_n$. For any  $0\leq s \leq d$ such that, for all $n\in\mathbb{N}$ large enough, $\mathcal{H}^{s}_{\infty}(U_n)\geq \mathcal{L}^d(B_n)$, it holds that
\begin{equation}
\dim_H (\limsup_{n\rightarrow+\infty}U_n)\geq s.
\end{equation}
\end{theoreme}

As a consequence of Theorem \ref{contss} and Corollary \ref{zzani}, one gets

\begin{corollary} 
\label{zzaniss}
Let $\mu\in\mathcal{M}(\R^d)$ be a self-similar measure and $\mathcal{B}= (B_n)_{n\in\mathbb{N}}$ be a   $\mu$-a.c. sequence of  closed balls of~$\R^d$ centered in $\supp (\mu)$.
Let $\mathcal{U}=(U_n)_{n\in\mathbb{N}}$ be a sequence of open sets such that 
$U_n \subset B_n$ for all $n\in \mathbb N$, and $0\leq s\leq \dim(\mu)$. If, for $n\in\mathbb{N}$ large enough, $\mathcal{H}^{\mu,s}_{\infty}(U_n)\geq \mu(B_n) $, then 
$$\dim_{H}(\limsup_{n\rightarrow +\infty}U_n)\geq s .
$$

\end{corollary}
One also emphasizes that  in the case of a self-similar measure, conversely, any $s\geq 0$ such that $\mathcal{H}^{\mu,s}_{\infty}(U_n)\leq \mu(B_n)$ for every $n$ large enough is an upper-bound for $\dim_H (\limsup_{n\rightarrow+\infty}U_n)$   if $\mathcal{B}$ verifies that, for any $p \in\mathbb{N}$, the balls $B_n$ with $\vert B_n \vert \approx 2^{-p}$ does not overlap too much. More precisely, in the companion paper of the present article, \cite{ED2}, the following result is proved.  
\begin{theoreme}[\cite{ED2}]
\label{majoss}
Let $\mu \in \mathcal{M}(\R^d)$ be a self-similar measure, $K$ its support and $(B_n)_{n\rightarrow+\infty}$ be a weakly redundant sequence of balls of $\R^d$ (see \cite[Definition 1.5]{BS2} ) verifying $\vert B_n \vert \to 0$ and, for any $n\in\mathbb{N}$, $B_n \cap K \neq \emptyset$. Let $(U_n)_{n\in \mathbb{N}}$ be a sequence of open sets satisfying $U_n \subset B_n$. For any  $0\leq s <\dim (\mu)$ such that, for all large enough  $n\in\mathbb{N}$, $\mathcal{H}^{\mu,s}_{\infty}(U_n)\leq \mu(B_n)$, it holds that
\begin{equation}\label{uppernound18}
\dim_H (\limsup_{n\rightarrow+\infty}U_n )\leq s.
\end{equation}

\end{theoreme}

Combining Theorem~\ref{lowani} and Corollary~\ref{minoeffec} with Theorem~\ref{contss} and Lemma~\ref{equiac} yield the following consequence for self-similar measures.

\begin{theoreme} 
\label{prop-ss}
Let $S$ be a self-similar  IFS of $\mathbb{R}^d$ with attractor $K$  and $\mu$ be a self-similar measure  associated with $S$. Let $(B_{n})_{n\in\mathbb{N}}$ be a sequence of closed balls centered on $K$, such that $\lim_{n\to+\infty} |B_{n}| = 0$.

\begin{enumerate}
\item Suppose that $(B_n)_{n\in\N}$ is $\mu $-a.c. Then $
t\big (\mu,\delta ,(B_n)_{n\in\mathbb{N}}\big )\le \dim(\mu)$; consequently $s\big (\mu, (B_n)_{n\in\N},(\widering B_n ^{\delta})_{n\in\N}\big ) \ge \frac{\dim(\mu)}{\delta}$ and there exists a gauge function~$\zeta$ such that $\lim_{r\to 0^+}\frac{\log(\zeta(r))}{\log(r)}\ge \frac{\dim(\mu)}{\delta}$ and $\mathcal H^\zeta(\limsup_{n\to\infty}  \widering B_n ^{\delta}) >0$.  In particular 
\begin{equation}\label{weakversion}
\dim_{H}(\limsup_{n\rightarrow+\infty}\widering B_n ^{\delta})\geq \frac{\dim(\mu)}{\delta} .
\end{equation}

\item Suppose that $\mu(\limsup_{n\rightarrow+\infty}B_n)=1$. Then,  \eqref{weakversion} still holds but the existence of the gauge function is not ensured. Furthermore if $\mu$ is doubling,  then $(B_n)_{n\in\N}$ is $\mu$-a.c, so that the conclusion of  item (1) holds.
\end{enumerate}
\end{theoreme}

%
%
%
%
%
%
%
%
%
%
%

\begin{remarque}
Since no separation condition is assumed about the system $S$, Theorem \ref{prop-ss} implies \cite[Theorem 1.6]{BS2} in the special case where the sequence of measures $(\mu_p )_{p\in\mathbb{N}}$ is  constant and equal to some self-similar measure with open set condition $\mu$ and the sequence of contraction ratio $(\delta_p)_{p\in\mathbb{N}}$ is constant as well. 
\end{remarque}

Corollary~\ref{zzani} and Theorem~\ref{contss}  also make it possible to deal with more general open sets $(U_n)_{n\ge 1}$ than the contracted balls $(\widering B_n^\delta)$, if one is able to compare efficiently the $s$-dimensional Hausdorff contents of the sets $U_n\cap K$ with a power of $|B_n|$. It is then convenient to assume that $K$ is the closure of its interior. Here is an example. 

Let $1\leq \tau_1 \leq ...\leq \tau_d$ be $d$ real numbers and $\tau=(\tau_1,...,\tau_d)$. One starts by defining a family of rectangles of $\mathbb{R}^d$ associated with $\tau$. 

 \begin{definition}
Let $1\leq \tau_1 \leq ... \leq \tau_{d} $ and $\tau=(\tau_1,...,\tau_d)$. For any $x=(x_i)_{1\leq i\leq d}\in \R^d$ and $r>0$, the  $\tau$-rectangle centered in $x$ and associated with  $r$ is defined by 
\begin{equation}
R_{\tau}(x,r)=\prod_{i=1}^d [x_i-\frac{1}{2}r^{\tau_i},x_i+\frac{1}{2}r^{\tau_i}].
\end{equation}  
\end{definition}
 \begin{theoreme}
 \label{rectss}
 Let $S$ be a self-similar IFS of $ \R^d$ such that the attractor $K$ is  equal to the closure of its interior. Let $\mu$ be a self-similar  measure associated with~$S$. Let $1\leq\tau_1\leq...\leq  \tau_d$, $\tau=(\tau_1, ...,\tau_d)$ and $(B_n :=B(x_n,r_n))_{n\in\mathbb{N}}$ be a sequence of balls of $\R^d$ satisfying $r_n \to 0$ and $\mu(\limsup_{n\rightarrow+\infty}B_n)=1.$ Define $R_n =\widering R_{\tau}(x_n,r_n).$ Then
 \begin{equation}
 \dim_H (\limsup_{n\rightarrow+\infty}R_n)\geq \min_{1\leq i\leq d }\left\{\frac{\dim(\mu)+\sum_{1\leq j\leq i}\tau_i-\tau_j}{\tau_i}\right\}.
 \end{equation}
 \end{theoreme}
\begin{remarque}
\label{remani}
(1)  Since $(\tau_1 ,... ,\tau_d)\mapsto \min_{1\leq i\leq d }\left\{\frac{\dim(\mu)+\sum_{1\leq j\leq i}\tau_i-\tau_j}{\tau_i}\right\}$ is continuous, the result stands for the sequence of closed rectangles as well. 

(2) One may also apply any rotation to the shrunk rectangles, this wouldn't change the bound (since Hausdorff contents are invariant by rotation).\mk
 
(3) Theorem \ref{rectss} extends the results of \cite{ED1}, where the measure was  quasi-Bernoulli or verifying the open set condition, and supported on $[0,1]^d$.\mk
 
(4) When $K$ is the closure of its interior  is that it is easy to compute $\mathcal{H}^s _{\infty}(R_n \cap K).$ Without this assumption, the conclusion of Theorem \ref{rectss} fails. Indeed, in general, no formula involving only the dimension of the measure and the contraction ratio can be accurate.  For instance, consider a self-similar measure in $\mathbb{R}^2$ carried by a line $D$ and a sequence $(B_n)_{n\in\mathbb{N}}$ of balls centered on the attractor $K$ and verifying $\mu(\limsup_{n\rightarrow+\infty}B_n)=1.$ Then,  consider the sequence of rectangles $R_n$ with side-length $\vert B_n \vert^{\tau_1}\times \vert B_n \vert^{\tau_2}$, $1\leq \tau_1 \leq \tau_2$ and where the largest side (of side-length $\vert B_n \vert^{\tau_1}$) is in the direction of $D$. In this case, Theorem~\ref{prop-ss} yields the lower-bound $\dim_H (\limsup_{n\rightarrow+\infty}R_n)\geq \frac{\dim_H (\mu)}{\tau_1}$. Then if $\widehat{R}_n$ are  the rectangles $R_n$ rotated by $\frac{\pi}{2}$, Theorem~\ref{prop-ss} gives that  $\dim_H (\limsup_{n\rightarrow+\infty}\widehat{R}_n)\geq \frac{\dim_H (\mu)}{\tau_2}$. Moreover, under additional conditions, these lower bounds are equalities. 
 
 \end{remarque}

\subsection{Application to self-similar shrinking targets}

We deal with points well approximable by orbits under an IFS with no exact overlap and satisfying the condition introduced by  Barral and Feng  (\cite{BFmult})  of being \textit{dimension regular} with similarity dimension less than than $d$, after Hochman's work (\cite{Hoc}). 

\label{sec-notaself}

\begin{definition}
\label{dimreg}
Let $S=\left\{f_1,...,f_m\right\}$ be a self-similar  IFS  of $\mathbb{R}^d$. Denote by $0<c_1,\ldots, c_m<1$ the contraction ratio of $f_1 ,\ldots,f_m$.  The system~$S$ is said to be dimension regular if, for any probability vector $(p_1,...,p_m)$, the self-similar measure associated with $S$ and the probability vector $(p_1,...,p_m)$ verifies $$\dim (\mu)=\min\left\{d,\frac{\sum_{1\leq i\leq m}p_i \log (p_i)}{\sum_{1\leq i\leq m} p_i \log(c_i)}\right\}.$$
This in particular, implies that, denoting by $\dim_{sim} (K)$ the unique real number $s$ satisfying $\sum_{i=1}^m c_i ^{s}=1,$ one has $\dim_H (K)=\min\left\{\dim_{sim}(K),d\right\}.$
\end{definition}

Some notation useful when dealing with IFS are introduced now. Those notations will be used repeatedly throughout this article. 

Let  $S=\left\{f_1 ,...,f_m\right\}$ be a self-similar IFS,  $0<c_1\ldots, c_m <1$ the associated contraction ratios, and $K$ the attractor of $S$.  Let $(p_1 ,...,p_m)$ be a  probability vector with positive entries, $\mu$ the self-similar associated with $S$ and $(p_1 ,...,p_m)$.

Let $\Lambda=\left\{1,...,m\right\}$ and $\Lambda^* =\bigcup_{k\geq 0} \Lambda ^{k}$. 
For  $k\geq 0$ and $\underline{i}:=(i_1 ,...,i_k)\in\Lambda^k$, define  
\begin{equation*}
\begin{split}
c_{\underline{i}}&=c_{i_1}\cdots c_{i_k},\  f_{\underline{i}}=f_{i_1}\circ \cdots\circ f_{i_k}, \\
\Lambda ^{(k)}&=\left\{\underline{i}=(i_1 ,\ldots, i_s) \in\Lambda^* : \ c_{i_s}2^{-k}< c_{\underline{i}}\leq 2^{-k}  \right\}.
\end{split}
\end{equation*}

\begin{theoreme}
\label{shrtag}
Let $S=\left\{f_1 ,...,f_m\right\}$ be a dimension  regular self-similar IFS with contraction ratio  $0<c_1,\ldots, c_m <1$ and such that the attractor $K$ verifies $\dim_{sim}(K)=\dim_H (K)$. For any $x\in K$, for any $\delta \geq 1$, 
\begin{equation}
\dim_{H}(\limsup_{\underline{i}\in \Lambda^*}B(f_{\underline{i}}(x), c_{\underline{i}}^{\delta}))=\frac{\dim_H (K)}{\delta}.
\end{equation}  
\end{theoreme}

This result extends some of the results obtained in \cite{AllenB} and  \cite{Baker}, under the open set condition.

\subsection{A result motivated by a question of  Mahler}$\ $
\label{sec-mahlint}

Let $\mathcal{Q}=\left\{B(\frac{p}{q},q^{-2})\right\}_{q\in\mathbb{N}^* ,\,  0\leq p\leq q}.$ Recall the following  result in Diophantine approximation \cite{Jarnik}:
\begin{align}
\label{approxrat}
&\bullet \limsup_{B\in\mathcal{Q}}B=[0,1]. \nonumber\\
&\bullet \text{ For any }\delta \geq 1, \dim_H (\limsup_{B\in\mathcal{Q}}B^{\delta})=\frac{1}{\delta}.
\end{align}
Unlike in the case of the points in $[0,1]$, the approximation by rational numbers of elements of the middle third Cantor set $K_{1/3}$ set is not well understood yet. This question was raised by Mahler, and only some partial results are known (see \cite{BV},~\cite{BS2}). Here we consider the set $K_{1/3}^{(0)}$ of points in $[0,1]$  having an asymptotic lower frequency of appearance of the digit $1$ in basis 3 equal to $0$.  This set contains $K_{1/3}$  and has the same Hausdorff dimension as $K_{1/3}$. We compute the Hausdorff dimension of sets of points in $K_{1/3}^{(0)}$ which are well approximable by rational numbers.


To describe more precisely the problem, let $S=\left\{f_1 ,f_2,f_3\right\}$ where $f_1$, $f_2$ and $f_3$ are  the contracting affine maps of $\mathbb R$ defined by $f_0 (x)=\frac{1}{3}x$, $f_1 (x)=\frac{1}{3}x+\frac{1}{3}$ and $f_2 (x)=\frac{1}{3}x+\frac{2}{3}$. The attractor of $S$ is  $[0,1]$. Let $\Lambda=\left\{0,1,2\right\}$.


 The  shift operation on the symbolic space $\Lambda^\mathbb{N}$ is defined by $\sigma$. The canonical projection from $\Lambda^\mathbb{N}$ to $[0,1]$ is the mapping  
\begin{equation}
\label{canoproj}
 \pi:x=(x_n)_{n\in\mathbb{N}}\mapsto \lim_{n\to +\infty} f_{(x_1 ,...,x_n )}(0).
\end{equation}

The set $K_{1/3}$ is the attractor of $\{f_0,f_2\}$ and also the image by canonical projection of $\{0,2\}^\mathbb{N}$.

\begin{definition}
\label{defphi}
Let  $\phi:\Lambda^\mathbb{N} \to \left\{0,1\right\}$ defined by $\begin{cases}\phi(x)=1 \text{ if } x_1 =1 \\
\phi(x)=0 \text{ if } x_1 =0 \text{ or }2.\end{cases}$

and 
\begin{align}
\label{defK}
 K_{1/3}^{(0)}=\pi\left(\left\{x\in \Lambda^\mathbb{N}: \liminf_{k\to +\infty}\frac{S_k\phi(x)}{k}=0 \right\}\right )\nonumber,
\end{align}
where $(S_k\phi)_{k\in\mathbb N}$ stands for the sequence of Birkhoff sums of $\phi$. 
\end{definition}
It is proved in  \cite{FanFengWu} that  $\dim_H K_{1/3}^{(0)}=\frac{\log 2}{\log 3}(=\dim_H K_{1/3})$.
Let us state the main results of this subsection.

\begin{theoreme}
\label{main}
For every $\delta\geq 1,$  
\begin{equation}
\label{mainequa}
\dim_H\Big (\limsup_{B\in\mathcal{Q}}B^{\delta}\cap K_{1/3}^{(0)}\Big )=\min\left\{\frac{\log 2}{\log 3},\frac{1}{\delta}\right\}.
\end{equation}
\end{theoreme}
%
Observe that a  saturation phenomenon occurs : $\dim_H(\limsup_{B\in\mathcal{Q}}B^{\delta}\cap K_{1/3}^{(0)})=\frac{\log 2}{\log 3}$ for $1\leq \delta \leq \frac{\log 3}{\log 2}$.


\bigskip

In Section \ref{sec-low}, the general ubiquity theorem, Theorem \ref{lowani}, is proved as well as Corollary \ref{minoeffec}. Section \ref{sec-parti} gives estimations of essential contents in the self-similar case and Theorem \ref{contss} is proved. 

Section \ref{sec-example} gives three applications to the main result, Theorem \ref{lowani}. More precisely, the ubiquity theorems for self-similar measures, Theorem \ref{prop-ss} and Theorem \ref{rectss}, are proved in the first sub-section. The second sub-section treats the case of self-similar shrinking targets for dimension regular IFS with similarity dimension less than $d$, e.g, Theorem \ref{shrtag} is proved. In the last sub-section, one gives an application in Diophantine approximation, Theorem \ref{main} is proved.  

\section{Proof of Theorem \ref{lowani}}
\label{sec-low}
\subsection{Preliminary facts}

We gather  in this subsection a series of results on which we will base the proof of Theorem \ref{lowani}.

The following lemma, which is a version of Besicovitch covering Lemma, as well as the subsequent one, both established in \cite{ED2}, will be used several times.
\begin{lemme} 
\label{besimodi}
For any $0<v\leq 1$ there exists $Q_{d,v} \in\mathbb{N}^{\star}$, a constant depending only on the dimension $d$ and $v$, such that for every bounded subset $E\subset \R^{d}$,  for every set $\mathcal{F}=\left\{B(x, r_{(x)} ): x\in E,  r_{(x)} >0 \right\}$, there exists $\mathcal{F}_1,...,\mathcal{F}_{Q_{d,v}}$ finite or countable sub-families of $\mathcal{F}$ such that:\medskip
\begin{itemize}

\item
$\forall 1\leq i\leq Q_{d,v}$, $L\neq L'\in\mathcal{F}_i$, one has $\frac{1}{v}L \cap \frac{1}{v}L'=\emptyset.$\medskip

\item 
$E$ is covered by the families $\mathcal{F}_i$, i.e.
\begin{equation}\label{besi}
 E\subset  \bigcup_{1\leq i\leq Q_{d,v}}\bigcup_{L\in \mathcal{F}_i}L.
 \end{equation}
\end{itemize}
\end{lemme}

\begin{lemme}[\cite{ED2}] \label{dimconst}
Let $0<v\leq 1$, $\mathcal{B} =(B_n )_{n\in\mathbb{N}}$ a family of balls,   and  $B$  a ball such that  
\medskip
\begin{itemize}
\item [(i)]
$\forall \ n\geq 1$, $\vert B_n\vert \geq \frac{1}{2}\vert B\vert,$

\item [(ii)]
 $\forall \ n_1 \neq n_2 \geq 1 $, $vB_{n_1}\cap vB_{n_2}=\emptyset$. 
\end{itemize}
Then $B$  intersects  less than $Q_{d,v}$ elements of $\mathcal{B}$, where $Q_{d,v}$ can be taken equal to the constant considered in  Lemma~\ref{besimodi}. 
\end{lemme}

The following lemma will also be useful later on and is also proved in \cite{ED2}.
\begin{lemme}
\label{distcov}
Let $\mathcal{L}$ be a family of pairwise disjoint balls satisfying $\sup_{L\in\mathcal{L}}\vert L\vert <+\infty.$ Then, for any $v\geq 1$, there exists sub-families $\mathcal{L}_1 ,...,\mathcal{L}_{Q_{d,v}}$ (where $Q_{d,v}$ is the constant of the same name in Lemma \ref{besimodi}) of $\mathcal{L}$ such that $\mathcal{L}=\bigcup_{1\leq i\leq Q_{d,v}}\mathcal{F}_i$ and for any $L\neq L^{\prime}\in\mathcal{L}_i$, $vL \cap vL^{\prime}=\emptyset.$
\end{lemme}

Recall the following version of Frostman Lemma, due to Carleson.
\begin{proposition}[\cite{Ca}]
\label{carl}
Let $s\geq 0$.  There is a constant $\kappa_d >0$ depending only on the dimension $d$ such that for any bounded set $E\subset \R^d$ with $\mathcal{H}^{s}_{\infty}(E)>0$, there exists  a probability measure supported by $E$, that we denote by $m^{s}_{E}$, such that 
\begin{equation} 
\label{majcarl}
\mbox{for every ball $ B(x,r)$, 
 } \ \ \ \  m_{E}^{s}(B(x,r))\leq \kappa_d \frac{r^s}{\mathcal{H}^{s}_{\infty}(E)}.
 \end{equation}
\end{proposition}

For  $s\geq 0$ and $E\subset \R^d$, a bounded subset such that $\mathcal{H}^{s}_{\infty}(E)>0$, $m^{s}_{E}$ will always denote such a measure associated with a (fixed) constant $\kappa_d$.

\sk

In the next two lemmas, the choice of the interval $[5,6]$ is convenient to take enough space between the shrunk balls involved in the construction elaborated in Section~\ref{construction}.


\begin{lemme} 
\label{doub2}

Let $t\in (5,6)$, $m\in\mathcal M(\R^d)$,  and $\ep  >0$.  Let $x\in \R^d$ be such that $\overline {\dim}_{\mathrm{loc}}(m, x) \leq \beta$. Let $C_{\beta,\varepsilon}= \frac{1}{2}6^{-\frac{\beta}{ 2 \ep  }}$. 
There exists an integer $n_{x}$ such that for every $ n\geq n_{x}$,  
\begin{equation}
\label{eq55}
\frac{\# \left\{0\leq k\leq n-1 : \ m(B(x,t^{-k-1}))\geq C_{\beta,\varepsilon} m(B(x,t^{-k}))\right\}}{n} \geq 1- \ep.
\end{equation}
\end{lemme}
Previous lemma is a slight extension of result by K\"aenm\"aki \cite[Lemma 2.2]{C}, which shows such a property at $m$-almost every point (where one has necessarily $\overline {\dim}_{\mathrm{loc}}(m, x) \leq d)$, and uses $t$ integer (a choice that we could make). 

Thus, points with a given local dimension with respect to a measure $m$ are for most scales ``locally doubling''. 

\begin{proof}
Observe first that if for a constant $0<C\leq 1$ and  some integer $n\in\mathbb{N}$ one has 
$$\frac{\#\left\{1\leq k\leq n : \ m(B(x,t^{-k-1}))\geq Cm(B(x,t^{-k}))\right\}}{n} \leq 1-\varepsilon,$$
then there necessarily exist $N = \lfloor (n-1)\ep \rfloor $ integers  $0 <k_1 <\cdots < W_{n}<n $  such that
 for every $ 1\leq i\leq N$,   $m(B(x,t^{-k_i-1}))\leq Cm(B(x,t^{-k_i}))$.

  In particular,   writing $k_{N+1}=n$ and $k_0 =0$, this implies that
\begin{align*}
m(B(x,t^{-n}))& =\prod_{i=0}^{N}  \frac{m(B(x,t^{-k_{i+1} }))}{m(B(x,t^{-k_i}))}\leq \prod_{i=0}^{N}  \frac{m(B(x,t^{-k_{i}-1 }))}{m(B(x,t^{-k_i}))}\leq C^{N} \\
& \leq   C^{(n-1)\ep  } \leq   C^{ n \ep/2   }  = (t^{-n})^{\ep \frac{-\log(C)}{2\log(t)}}.
\end{align*}
The inequality $C^{(n-1)\ep  } \leq   C^{ n \ep/2   } $ occurs when $n$ is large enough. Recalling that $\overline{\dim}_{\mathrm{loc}}(m,x)\leq\beta$, if this happens for infinitely many $n$, one should have 
$$\beta\geq \limsup_{r\rightarrow 0^+}\frac{\log m(B(x,r))}{\log r}\geq  \ep\frac{-\log(C)}{2\log(t)},$$
which is equivalent to  $C\ge t^{-\frac{\beta}{2 \ep }}$.

 Setting $C_{ \ep ,\beta}=\frac{1}{2}6^{-\frac{\beta}{ 2 \ep  }}$, one concludes that  there exists $n_{x}$ such that for every $ n\geq n_{x}$, one necessarily has
$$\frac{\#\left\{0\leq k\leq n-1 : \ m(B(x,t^{-k-1}))\geq C_{ \ep ,\beta } m(B(x,t^{-k}))\right\}}{n} \geq 1- \ep ,$$ 
hence the result.
\end{proof}

\begin{lemme} 
\label{doub}
Let $m$ and $\mu$  be two elements of $\mathcal M(\zu)$, $\beta \geq 0$ and $\varepsilon >0$. For every $x\in \R^d$ verifying $\overline {\dim}_{\mathrm{loc}}(m, x)\leq \beta$, there exists $\rho_x >0$ and $t_x\in (5,6)$ so that for all $0<r\leq \rho_x$ there exists $r\leq r'\leq r^{1-\varepsilon }$   such that 
\begin{equation}
\label{eqdoub}
m(B(x,r'/t_x))\geq C_{\beta,\frac{\ep}{2}}m(B(x,r')) \mbox{ and } \mu(\partial B(x,r'/t_x))=0.
\end{equation}
\end{lemme}

\begin{proof}
Consider $x\in \R^d$ such that $\overline{\dim}_{\mathrm{loc}}(m,x)\leq\beta.$

We apply Lemma \ref{doub2} to $x$ and the measure $m$, and for an arbitrary $t\in [5,6]$ and $\varepsilon^{\prime}=\frac{\varepsilon}{2}$: for $n\geq n_{x}$, there must be  an integer $n'$ such that  $n(1-\ep) \leq n'\leq n $ and $m(B(x,t^{-n'-1}))\geq C_{\beta,\frac{\varepsilon}{2}}m(B(x,t^{-n'})).$

Let $\rho_x=\min\left\{t^{-n_x-1 },t^{-\frac{1}{\varepsilon}}\right\}$. For $r\in(0,\rho_x]$, let   $n$ be the integer such that $t^{-n-1}<r\leq t^{-n}$. The previous claim yields an integer $n'\in [n(1-\frac{\ep}{2}),n] $  such that $m(B(x,t^{-n'}))\geq C_{\beta,\frac{\ep}{2}} m(B(x,t^{-n'+1}))$. Also, 
$$ r \leq r'=t^{-n'+1 }\leq t^{1-(1-\frac{\ep}{2} )n}= t^2 \cdot t^{-n-1}\cdot t^{ \frac{\ep}{2}  n}\leq t^2 \cdot r \cdot r^{-\frac{\ep}{2} }\leq r^{1- \ep }.$$
Consequently, 
$$
m(B(x,r'/t))\geq  C_{\beta,\frac{\ep}{2}}m(B(x,r')).
$$
The desired conclusion holds if we choose $t_x\in (5,t)$ such that $\mu(\partial B(x,r'/t_x))=0$. 
\end{proof}
The previous  lemma will be used in the case $\beta=d$ in our proof the main theorem (see step 2 of the construction in Section~\ref{construction}).

\mk

Next, we introduce some some sets  associated to a given element of $\mathcal M(\R^d)$, which will play a natural role in our construction. 

\begin{definition} 
\label{emudef}
 {Let $\beta\geq\alpha\geq 0$ be  real numbers, $m\in\mathcal{M}(\R^d)$, and $\varepsilon,\rho>0$ two positive real numbers. Then define }
\begin{equation}
\label{emuhat}
 {\widetilde{E}_{m}^{[\alpha,\beta],\rho,\varepsilon}=\left\{x\in\mathbb{R}^d : \ \underline\dim_{{\rm loc}}(m,x)\in[\alpha,\beta] \ and \  \ \forall r\leq \rho, \  m(B(x,r))\leq r^{\underline{\dim}_{\mathrm{loc}} (m,x)-\varepsilon}\right\}}
\end{equation} 
 and
\begin{equation}
\label{emu}
 {E_{m}^{[\alpha,\beta],\rho,\varepsilon}=\left\{ x\in\widetilde{E}_{m}^{[\alpha,\beta],\rho,\varepsilon}: \forall r\leq \rho, \  \frac{3}{4}m(B(x,r)) \leq m(B(x,r)\cap \widetilde{E}_{m}^{[\alpha,\beta],\rho,\varepsilon})\right\}}
\end{equation}

\end{definition}
 {Notice   that, for every $0 <\rho <\rho '$, one  has   $ E_{m}^{[\alpha,\beta],\rho',\varepsilon }\subset E_{m }^{[\alpha,\beta],\rho,\varepsilon }.$}

\begin{definition}
\label{emudef2}
 {Let $\beta \geq\alpha\geq 0$ be  real numbers, $m\in\mathcal{M}(\R^d)$, and $\varepsilon>0$.  Define}
\begin{equation} 
\label{emut}
 {E_{m}^{[\alpha,\beta],\varepsilon}=\bigcup_{n\geq 1}E_{m}^{[\alpha,\beta],\frac{1}{n},\varepsilon}} .
\end{equation}
\end{definition}

\begin{proposition}
\label{mle}
{For every $m\in\mathcal{M}(\R^d)$, every $\beta\geq \alpha\geq 0$ and  $\varepsilon>0$,}
\begin{equation}
{m(E_{m}^{[\alpha,\beta],\varepsilon})=m(\left\{x:\underline{\dim}_{\mathrm{loc}}(m,x)\in [\alpha,\beta]\right\})}.
\end{equation}
\end{proposition}


{Notice   that, for every $0<\rho ' <\rho$, one  has   $E_{m }^{[\alpha,\beta],\rho,\varepsilon }\subset E_{m}^{[\alpha,\beta],\rho',\varepsilon }.$}

 
{ These sets  play a key role in the proofs  of Theorem \ref{lowani} .}
 \begin{proof}
{One first recalls the following result. }
 \begin{lemme}\cite{Be}
 \label{densibesi}
 {Let $m\in\mathcal{M}(\R^d)$ and $A$ be a Borel set with $m(A)>0.$ For every $r>0$, set}
 \begin{equation} 
 \label{defniv} 
 {A(r) =\left\{x\in A \ : \ \forall \tilde{r}\leq r, \ m(B(x,\tilde{r})\cap A)\geq \frac{3}{4}m(B(x,\tilde{r}))\right\}} 
 \end{equation}
 {Then }
 \begin{equation}
{ m\left(\bigcup_{r>0}A(r)\right)=m(A).}
 \end{equation}
 \end{lemme}
 Note that it is clear from Definition \ref{dim} that 
 $$\left\{x:\underline{\dim}_{\mathrm{loc}}(m,x)\in [\alpha,\beta]\right\}=\bigcup_{\rho>0} \widetilde{E}_{m}^{[\alpha,\beta],\rho,\varepsilon}.$$
 
{ Let $\ep'>0$. By Definition \ref{dim}, there exists $\rho_{\ep'}$ small enough so that} 
\begin{equation}
\label{eqvo}
{m(\widetilde{E}_{m}^{[\alpha,\beta],\rho_{\ep'}, \varepsilon})\geq (1-\ep')m(\left\{x:\underline{\dim}_{\mathrm{loc}}(m,x)\in [\alpha,\beta]\right\}).}
\end{equation} 
 {By Lemma \ref{densibesi} (and the notations therein) applied to $\widetilde{E}_{m}^{[\alpha,\beta],\rho_{\ep'}, \varepsilon}$, there exists $\tilde{\rho}_{\ep'}$ such that} 
 \begin{equation}
 \label{eqvo2}
{ m(\widetilde{E}_{m}^{[\alpha,\beta],\rho_{\ep'}, \varepsilon}(\tilde{\rho}_{\ep'}))\geq (1-\ep')m(\widetilde{E}_{m}^{[\alpha,\beta],\rho_{\ep'}, \varepsilon}) .} 
 \end{equation}
{Finally for $\rho =\min\left\{\rho_{\ep'},\tilde{\rho}_{\ep'}\right\}$,  by Definition \ref{emudef} and \eqref{defniv}, one has  $ (\widetilde{E}_{m}^{[\alpha,\beta],\rho_{\ep'}, \varepsilon})_{\tilde{\rho}_{\ep'}}\subset E_{m}^{[\alpha,\beta],\rho, \varepsilon}$, so that, by \eqref{eqvo} and \eqref{eqvo2}}
 \begin{align*}
  m(E_{m}^{[\alpha,\beta],\rho, \varepsilon})
&\geq m((\widetilde{E}_{m}^{[\alpha,\beta],\rho_{\ep'}, \varepsilon}(\tilde{\rho}_{\ep'}))
 \geq (1-\ep')m(E_{m}^{[\alpha,\beta],\varepsilon})\\
 &\geq (1-\varepsilon^{\prime})^2 m(\left\{x:\underline{\dim}_{\mathrm{loc}}(m,x)\in [\alpha,\beta]\right\}).
 \end{align*} 
In particular
 $$m(\left\{x:\underline{\dim}_{\mathrm{loc}}(m,x)\in [\alpha,\beta]\right\})\geq m(E_{m}^{[\alpha,\beta],\varepsilon}) \geq (1-\varepsilon^{\prime})^2 m(\left\{x:\underline{\dim}_{\mathrm{loc}}(m,x)\in [\alpha,\beta]\right\}).$$
{Letting $\ep' \to 0$ proves the result.}
 \end{proof}
 \begin{corollary}
 \label{mlecoro}
For every $m\in\mathcal{M}(\R^d)$, for $\alpha=\underline{\dim}_H (m)$ and $\beta=\overline{\dim}_H (m)$, for any  $\varepsilon>0$, one has
\begin{equation}
m(E_{m}^{[\alpha,\beta],\varepsilon})=1.
\end{equation}
 \end{corollary}

%

%




\subsection{Construction of the Cantor set and the measure}\label{construction}

  Recall that $\mu$ is a probability measure on $\R^d$, and that $\mathcal{B} =(B_n :=B(x_n ,r_n))_{n\in\mathbb{N}}$ is a $\mu$-a.c sequence of balls of $ \R^d$ with $\lim_{n\to +\infty} r_{n}= 0.$ Fix $\mathcal{U}=(U_n)_{n\in\mathbb{N}}$ a sequence of open sets satisfying $U_n \subset B_n$ for every $n\in\mathbb{N}$. 
 
Set $\alpha=\dimm(\mu)$, and assume that $\min\left\{ s(\mu,\mathcal{B},\mathcal{U}),\alpha\right\}>0$. 

Our goal is to  construct a gauge function $\zeta:\R^+\to \R^+$ such that $\lim_{r\to 0^+} \frac{\log \zeta(r)}{\log r} =\min\left\{ s(\mu, \mathcal{B},\mathcal{U}),\dimm(\mu)\right\}$ as well as $\eta\in \mathcal M(\R^d)$ supported on $\limsup_{n\to\infty} U_n$ such that  for all $r\in(0,1]$ and $x\in\R^d$ one has $\eta(B(x,r))\le \zeta(2r)$.

Let $(\varepsilon_k)_{k\in\mathbb{N}}$ be a sequence decreasing to 0 and such that $\varepsilon_1<s(\mu,\mathcal{B},\mathcal{U})$.   For $k\geq 0$, set
\begin{equation}
\label{defskani}
s_k=\min\left\{ s(\mu,\mathcal{B},\mathcal{U}),\alpha\right\}-\varepsilon_{k}.
\end{equation}

Along the construction of $\zeta$, we only use that $s_k < s(\mu,\mathcal{B},\mathcal{U})$ and the fact that $s_k < \alpha$  is  used at the end of our analysis (see equation \eqref{res3ani}).

\subsubsection*{Step 1} We need the following lemma.
\begin{lemme} [\cite{ED2}]
\label{covO}
Let   $\mu\in\mathcal{M}(\R^d)$  and $\mathcal{B} =(B_n :=B(x_n ,r_n))_{n\in\mathbb{N}}$ be a $\mu$-a.c sequence of balls of $ \R^d$ with $\lim_{n\to +\infty} r_{n}= 0$.

Then  for every open set $\Omega$ and every integer $g\in\mathbb{N}$, there exists  a subsequence  $(B_{(n)}^{(\Omega)})\subset \left\{B_n \right\}_{n\geq g} $ such that:
\begin{enumerate}
\item 
$\forall \,  n\in\mathbb{N}$, $B_{(n)}^{(\Omega)}\subset \Omega,$
\sk
\item
 $\forall \, 1\leq n_1\neq n_2$, $B_{(n_1)}^{(\Omega)}\cap B_{(n_2)}^{(\Omega)}=\emptyset$,
 \sk
\item
 $\mu\left (\bigcup_{n\geq 1}B_{(n)}^{(\Omega)}\right)= \mu(\Omega).$
\end{enumerate}
In addition, there exists an integer $N_{\Omega}$ such that  for the balls $ (B_{(n)}^{(\Omega)})_{n=1,...,N_\Omega}$, the conditions (1) and (2) are realized, and (3) is replaced by  $\mu\left (\bigcup_{n=1}^{N_\Omega} B_{(n)}^{(\Omega)}\right)\geq \frac{3}{4} \mu(\Omega).$  
\end{lemme}

The last part of Lemma \ref{covO} simply follows from item (3) and the $\sigma$-additivity of $\mu$.

 Using Lemma \ref{covO} with, $(B_n)_{n\in \mathcal{N}_\mu(\mathcal{B},\mathcal U,s_1)}$ (which is $\mu$-a.c since $ s_1 <s(\mu,\mathcal{B},\mathcal{U})$), $g=0$  and $\Omega = \R^d$,   one finds integers $N_1 $ and $n_1 <  ... < n_{N_1}\in\mathcal{N}_\mu(\mathcal{B},\mathcal U,s_1)$ such that :
 \begin{itemize}
\item[$(i):$] $\forall \,  1\leq i\leq N_1$, $B_{n_i}\cap B_{n_j}=\emptyset$,
\item[$(ii):$] $\mu(\bigcup_{1\leq i\leq N_1}B_{n_i})\geq \frac{1}{2}$.

\end{itemize}

By Lemma \ref{distcov} applied to $\left\{B_{n_i}\right\}_{1\leq i\leq N_1}$ and $v=4$, the balls $\left\{B_{n_i}\right\}_{1\leq i\leq N_1}$ can be sorted in $Q_{d,4}$  families of balls $\mathcal{L}_1 ,...,\mathcal{L}_{Q_{d,4}}$ such that 
\begin{itemize}
\item[•] for any $1\leq i\leq Q_{d,4}$, any $L\neq L^{\prime}\in\mathcal{L}_i$, $4L \cap 4L^{\prime}=\emptyset, $\mk
\item[•] $\bigcup_{1\leq i\leq Q_{d,4}}\mathcal{L}_i=\left\{B_{n_i}\right\}_{1\leq 1\leq N_1}.$
\end{itemize}

At least one of these families, $\mathcal{L}_{i_0 }$, must satisfy
$$
\mu\big(\bigcup_{L\in\mathcal{L}_{i_0}}L\big)\geq \frac{1}{2Q_{d,4}}.
$$

In particular, if one must rename the balls of the family $\mathcal{L}_{i_0}$, we can assume that the family $\left\{B_{n_i}\right\}_{1\leq i\leq N_1}$ satisfies
\begin{itemize}
\item[$(i^{'}):$] for any $1\leq i< j\leq N_1$, $4 B_{n_i}\cap 4B_{n_j}=\emptyset$\mk
\item[$(ii^{\prime}):$] and
\begin{equation}
\label{recouetape1}
\mu\big(\bigcup_{1\leq i\leq N_1}B_{n_i}\big)\geq \frac{1}{2Q_{d,4}}.
\end{equation}
\end{itemize}

Set $$\mathcal{W}_{1}=\left\{U_{n_i} \right\}_{1\leq i\leq N_1}  \ \ \mbox{ and } \ \ W_{1}=\bigcup_{1\leq i\leq N_1}U_{n_i} .$$

Along the construction of the Cantor set, for every $U \in\mathcal{U}$, the ball of $\mathcal{B}$ naturally associated with $U$ will be denoted $B^{[U]}$ (that is $B^{[U_n]}=B_n$). 

The pre-measure $\eta$ on the $\sigma$-algebra generated by the  sets  of $\mathcal{W}_{1} $ is defined by   
\begin{equation}
\label{defeta0}
\mbox{ for  every $U \in \mathcal{W}_1$, } \ \ \ \ \eta(U )=\frac{\mu(B^{[U]})}{\sum_{\widetilde{ U} \in \mathcal{W}_1}\mu(B^{[\widetilde{U}]})} . 
\end{equation}
It is obvious that $\eta(\R^d) = \eta (W_{1}) =1$.

  Recalling \eqref{defNsdani} and \eqref{defsmdani}, since $s_1<s(\mu , \mathcal{B},\mathcal{U})$,   the sub-sequence   $(B_n )_{n\in\mathcal{N}_{\mu}(\mathcal{B},\mathcal{U},s_1 )} $ is $\mu$-a.-c.  Recall  also that $\lim_{n\to +\infty}r_n=0$ and for every  $n\in \N$,  $\vert U_n \vert \leq r_n.$

  So,  for every $ n\in\mathcal{N}_{\mu}(\mathcal{B},\mathcal{U},s_1)$,
  \begin{equation}
  \label{eqmunani}
 \mathcal{H}^{\mu,s_1}_{\infty}\left (U_n \right)\geq \mu( B_n ) \   \ \ \ \ \ \  \mbox{ and } \ \ \ \vert U_n \vert\leq r_n.
  \end{equation}
In particular, by Definition \ref{mucont}, for every $n\in\mathcal{N}_{\mu}(\mathcal{B},\mathcal{U},s_1) $ for any set $E_n \subset U_n $ with $\mu(E_n)=\mu(U_n)$, 
$$
\mu(B_n)\leq \mathcal{H}^{\mu,s_1}_{\infty}(U_n)\leq \mathcal{H}^{s_1} _{\infty}(E_n).
$$
By Lemma \ref{carl}, and the notations therein,  one has 
$$
m_{E_n}^{s_1}  (U_n)=1 \leq \frac{\kappa_d \vert U_n \vert ^{ {s_1} }}{\mathcal{H}^{{s_1} }_{\infty}(E_n)}\leq \frac{\kappa_d \vert U_n \vert ^{ {s_1} }}{\mu(B_n)}.
$$
This implies that 
\begin{equation}
\label{bonnemajoani}
\mu(B_n)\leq \kappa_d \vert U_n \vert ^{ {s_1} }.
\end{equation}

By equation  \eqref{bonnemajoani}, recalling the fact that the sets $ \mathcal{W}_1 \subset \left\{U_n\right\}_{n\in\mathbb{N}}$ , one has for every  $U\in \mathcal{W}_1 ,$
\begin{equation}
\label{eq12}
\eta(U )\leq \frac{\mu(B^{[U]})}{\frac{1}{2Q_{d,4}} }\leq  2Q_{d,4}\kappa_d  |U | ^{s_1  }. 
\end{equation}

\subsubsection*{Step 2} This step (and all the following steps) is split    into two sub-steps. First,  into each open set $U$ of $\mathcal{W}_1$,   smaller intermediary balls are selected according to the $\mu$-essential content of $U$. Then in a second time, each intermediary ball  will be covered by balls of the sequence $(B_n)_{n\in\mathbb{N}}$ according to the measure $\mu$    and, as in step 1, the  sets $U_n$ associated with this covering will form the generation $\mathcal{W}_2$.

\sk

Let $g\in \N$ be such that for every $n\geq g$, $r_n\leq \frac{1}{3}\min(|U|:U\in \mathcal{W}_1)$.

As above,  since $s_2<s(\mu , \mathcal{B},\mathcal{U})$,   the sub-sequence   $(B_n )_{n\in\mathcal{N}_{\mu}(\mathcal{B},\mathcal{U},s_2), n\geq g  } $ is $\mu$-a.c.    The same arguments as above yield  for every $ n\in\mathcal{N}_{\mu}(\mathcal{B},\mathcal{U},s_2) $,
  \begin{equation}
  \label{eqmun2ani}
 \mathcal{H}^{\mu,s_2}_{\infty}\left (U_n\right)\geq \mu( B_n ) \    \ \ \ \ \ \  \mbox{ and } \ \ \ \vert U_n \vert \leq r_n
  \end{equation}
and  
\begin{equation}
\label{bonnemajo2ani}
\mu(B_n)\leq \kappa_d \vert U_n \vert^{ {s_2} }.
\end{equation}

\subsubsection*{Covering with respect to the $\mu$-essential content} 
 
Consider $U \in\mathcal{W}_{1}$. Set  $\beta=\overline\dim_H(\mu)$. For $0\leq k\leq \lfloor \frac{\beta-\alpha}{\varepsilon_2}\rfloor+1$, define $\theta_k =\alpha+k\varepsilon_2 $. 
Write  
\begin{equation}
\label{defELani}
E_{U}= U\cap E_{\mu}^{[\alpha,\beta],\varepsilon_2} \cap \limsup_{n\to +\infty }B_n.
\end{equation}
 Notice that by Proposition \ref{mle} and by item (1) of Lemma~\ref{equiac}, one has $\mu(E_{U})=\mu(U)$. 
 
 In addition, using the definition \eqref{eqmucont} of $ \mathcal{H}^{\mu,s_2}_{\infty}$, the fact that  $E_U\subset U$ and $\mu(E_{U})=\mu(U)$, and finally \eqref{defsndani} applied with $B_n = B^{[U]}$, one gets 
\begin{equation}
\label{eq21ani}
\mathcal{H}_{\infty}^{s_2}(E_U)\geq \mathcal{H}^{\mu,s_2}_{\infty}(U)\geq \mu(B^{[U]})>0.
\end{equation}

 This allows us to apply Proposition \ref{carl}: there exists  a Borel probability measure $m_{E_U}^{s_2}$ supported on $E_{U}$ such that for every ball $B:=B(x,r)$, one has $$m_{E_U}^{s_2}(B)\leq \kappa_d \frac{r^{s_2}}{ \mathcal{H}^{s_2}_{\infty}(E_U)}.$$ 
 
 Also, since  $m_{E_U}^{s_2}(E_U)=1$ and $E_U \subset  E_{\mu}^{[\alpha,\beta],\varepsilon_2}$, and recalling \eqref{emut}, for any $0\leq k\leq \lfloor \frac{\beta-\alpha}{\varepsilon_2}\rfloor+1$,  there exists $\rho_{k,\varepsilon_2}$ such that 
 $$m_{E_U}^{s_2}(E^{[\theta_k,\theta_{k+1}],\rho_{k,\varepsilon_2},\varepsilon_2}_{\mu})\geq \frac{1}{2}m_{E_U}^{s_2}(E^{[\theta_k,\theta_{k+1}],\varepsilon}_{\mu}).$$

  Setting   $\rho_U =\min_{0\leq k\leq \lfloor \frac{\beta-\alpha}{\varepsilon_2}\rfloor+1}\rho_{k,\varepsilon_2}$ one has, for any $0\leq k\leq \lfloor \frac{\beta-\alpha}{\varepsilon_2}\rfloor +1$, 
 \begin{equation}
 \label{equarecti}
 m_{E_U}^{s_2}(E^{[\theta_k,\theta_{k+1}],\rho_{U},\varepsilon_2}_{\mu})\geq \frac{1}{2}m_{E_U}^{s_2}(E^{[\theta_k,\theta_{k+1}],\varepsilon}_{\mu}).
 \end{equation}
 In particular,
\begin{equation}
m_{E_U}^{s_2}(E_{\mu}^{[\alpha,\beta],\rho_U ,\varepsilon_2})\geq \frac{1}{2}
\end{equation}

\mk

Let
\begin{equation}
\label{defSLani}
S_U := \bigcup_{0\leq k\leq \lfloor\frac{\beta-\alpha}{\varepsilon_2}\rfloor+1}E^{[\theta_k,\theta_{k+1}],\rho_{U},\varepsilon_2}_{\mu}\cap E_U \cap \left\{x\in \R^d : \ \overline{\dim}_{\mathrm{loc}}(m_{E_U}^{s_2} ,x)\leq d\right\}.
\end{equation}
Recalling that for every probability measure $m$, $m(\{x=  \overline{\dim}_{\mathrm{loc}}(m  ,x)\leq d\})=1$, one necessarily has $m_{E_U}^{s_2}(S_U) \geq 1/2$. 

Let $x\in S_U $; consider $0\leq k_x\leq \lfloor \frac{\beta-\alpha}{\varepsilon_2}\rfloor+1$ such that   $x\in E^{[\theta_{k_x},\theta_{k_x+1}],\rho_{U},\varepsilon_2}_{\mu}.$ Applying Lemma \ref{doub}, there exists  $0<r_{x}<\min \big (\rho_x,\frac{1}{3}\min\left\{\vert V\vert:V\in\mathcal{W}_1 \right\}\big) $ and $t_x\in (5,6)$  such that: 
\begin{align}
\label{recap-ani}
 &10\,  r_x <\rho_{U};\\
\label{recap0ani}
& B(x,r_x)\subset U  \ \  \mbox{ and } \ \ \mu(\partial B(x,r_x/t_x))=0;\\
 \label{recap1ani}
&  r_x^{-\ep_2} \geq 5^{d} \frac{4 Q_{d,1}} {C_{\ep_3,d}} \frac{\eta(U)}{\mu(B^{[U]}) }\geq 5^{s_2} \frac{4 Q_{d,1}} {C_{\ep_2,d}} \frac{\eta(U)}{\mu(B^{[U]}) };\\
 \label{recap1bisani}
 &    r_{x} ^{\theta_{k_x}+2\varepsilon_2}\leq\mu(B(x,r_{x}))\leq  r_{x} ^{\theta_{k_x}-2\varepsilon_2};\\
\label{recap2ani}
& m_{E_U}^{s_2}(B(x,r_x/t_x))\geq  C_{\ep_2,d} \cdot m_{E_U}^{s_2}(B(x,r_x)).
\end{align}
Note that in \eqref{recap1ani} the second inequality follows automatically from the first one since $s_2\le \alpha\le d$ and the constant $C_{\ep,d}$ is an increasing function of $\varepsilon$.

The family $\left\{B(x,r_x) : x\in S_U\right\}$ forms a covering of $S_U$. We apply Lemma~\ref{besimodi} with $v=1$ (i.e., the standard Besicovich  covering Theorem) to this family to  extract  $Q_{d,1} $ subfamilies of balls, $\mathcal{G}^{U}_1 ,...,\mathcal{G}^{U}_{Q_{d,1} }$ such that: 
\begin{itemize}
\sk
\item $\forall 1\leq i \leq Q_{d,1} $, $\forall B \neq B' \in\mathcal{G}^{U}_{i}$, one has $B \cap B' =\emptyset,$ 
\sk
\item $S_U \subset \bigcup_{i=1}^{Q_{d,1} }\bigcup_{B \in\mathcal{G}^{U}_i}B.$ 
\end{itemize}
In particular, $m_{E_U}^{s_2} \left(\bigcup_{i=1}^{ Q_{d,1} }\bigcup_{B\in\mathcal{G}^{U}_i}B \right) \geq  m_{E_U}^{s_2}(S_U) \geq 1/2 $.

At least one of these families, say $\mathcal{G}^{U}_{i_0}$, verifies  that 
$$m_{E_U}^{s_2}\left (\bigcup_{B \in\mathcal{G}^{U }_{i_0}} B  \right )\geq \frac{m_{E_U}^{s_2}(S_U)}{Q_{d,1}}\geq \frac{1}{2Q_{d,1}}.$$
Writing $ \mathcal{G}^{U}_{i_0}=\left\{B_{i_0,k}^{U}\right\}_{k\in\mathbb{N}}$, one can find an integer  $N_U $  so large     that   
$$m_{E_U}^{s_2}\left(\bigcup_{1\leq k \leq N_U}B_{i_0,k}^{U}\right)\geq \frac{1}{4 Q_{d,1}}. $$

Remind that each $B_{i_0,k}^{U}$ is a ball $B(x,r_x)$ satisfying \eqref{recap0ani},  \eqref{recap1ani} and \eqref{recap2ani}.

Finally, setting  $\mathcal{G}^{U}=\left\{B(x, r_x/t_x): B(x,r_x) \in \mathcal{F}^{U }_{i_0} \right\}$, one has by construction
\begin{equation}
\label{eq22ani}
m_{E_U }^{s_2}\left(\bigcup_{B\in  \mathcal{G}^{U} }B \right) =  \sum_{B\in  \mathcal{G}^{U} } m_{E_U }^{s_2}(B) \geq \frac{C_{\ep_2,d}}{4 Q_{d,1}} .
\end{equation}

One then extends the pre-measure $\eta$ to the Borel $\sigma$-algebra generated by the balls of $\mathcal{G}^{U}$, by the formula 
\begin{equation}
\label{etaSani}
\mbox{ for every $B \in \mathcal{G}^{U}$,  } \ \ \ \eta (B )=\eta(U)\times  \frac{m_{E_U}^{s_2}(B )} {\sum_{B' \in \mathcal{G}^{U}}m_{E_U}^{s_2}(B')}.
\end{equation}
By construction, this formula is consistent  since  $\eta(U) = \sum_{B\in \mathcal{G}^{U}} \eta (B)$.

\mk

Observe that by \eqref{majcarl}, \eqref{eq22ani} and \eqref{eq21ani}, one has for every $B \in \mathcal{G}^{U}$,
\begin{align}
\label{eq23ani}
\eta (B ) & \leq \eta(U)  \kappa_d \frac{|B|^{s_2}}{\mathcal{H}^{s_2}_{\infty}(E_U)   } \frac{4 Q_{d,1}}{C_{\ep_2,d}} \leq  \frac{4 Q_{d,1} \kappa_d }{C_{\ep_2,d}}    \frac{ \eta(U)  } {\mu(B^{[U]})} |B|^{s_2}  \leq   |B|^{s_2-\ep_2}   ,
\end{align} 
where the second inequality of \eqref{recap1ani} was used.

\mk

This is achieved simultaneously for all $U\in\mathcal{W}_{1}$.

\mk

\subsubsection*{Covering with respect to $\mu$}  Now, in order to build the second generation of the Cantor set $K$, we select balls of $\mathcal{B}$ that lie  in the interior of these intermediate balls  $B\in \mathcal{G}^{U}$.

Let  $U\in\mathcal{W}_{1}$ and $B\in\mathcal{G}^{U}$ be one of these intermediary balls. Since $\mathcal{B}$ is $\mu$-a.c.,  the last part of Lemma \ref{covO} proves the existence of  a finite family  $ \mathcal{F}^{B}=\left\{U_{n_i}\right\}_{ 1\leq i \leq  {N_B}}$ such that  
\begin{itemize}
\sk
\item[($i_1$)] for every $ 1\leq i\leq  {N_B }$, one has $B_{n_i}\subset \widering{B} $ and 
\begin{equation}
\label{eq24ani}
\max\left\{ 2 Q_{d,4}   \frac{\eta(B)}{\mu(B)},\frac{5^d 4Q_{d,1} \kappa_d }{C_{\varepsilon_3 ,d}}\right\} \leq   {r_{n_i }^{-\varepsilon_2}} ,
\end{equation}
\sk
\item[($i_2$)]   for every $ 1\leq i\neq j\leq  {N_B}$, one has $B_{n_i}\cap B_{n_j}=\emptyset.$   
\end{itemize}

In addition, recalling that $\mu(\partial B)=0$ by \eqref{recap0ani}, one has
\begin{equation*}
\mu(B_{n_i}) >0 \ \ \mbox{ and } \ \ \mu\Big (\bigcup_{1\leq i\leq  {N_B}}B_{n_i} \Big)\geq \frac{3\mu(\widering{B})}{4}=\frac{3\mu(B)}{4}.
\end{equation*}

 {Recall the  definitions \eqref{emu}  and \eqref{defSLani} of the sets $E_{\mu}^{[a,b],\rho_{U},\varepsilon_2}$ and $S_{U}$.  By  equations  \eqref{recap-ani}-\eqref{recap2ani}, there exists $\alpha\leq a\leq \beta$ such that the center of $B$ belongs to $S_U\subset E_{\mu}^{[a,a+\varepsilon_2],\rho_U \varepsilon_2,}$ and $\vert B\vert \leq \rho_U$, hence one has}
$$ {\mu(B\cap \widetilde{ E}_{\mu}^{[a,a+\varepsilon_2],\rho_{U},\varepsilon_2})\geq \frac{3}{4}\mu(B).}$$
 {By $(i_2)$, and recalling   \eqref{emu}, one has

\begin{align*} 
 &{\mu\Big (\bigcup_{B_{n_i} : B_{n_i} \cap \widetilde{ E}_{\mu}^{[a,a+\varepsilon_2],\rho_U ,\varepsilon_2}\neq \emptyset }  B_{n_i}\Big)}
  \geq \mu\Big(\bigcup_{1\leq i\leq N_B}B_{n_i}\cap  \widetilde{ E}_{\mu}^{[a,a+\varepsilon_2],\rho_U ,\varepsilon_2}\Big)\\
&=\mu\left(\bigcup_{1\leq i\leq N_B}B_{n_i}\right)+\mu\left(\widetilde{ E}_{\mu}^{[a,a+\varepsilon_2],\rho_U ,\varepsilon_2}\right)- \mu\left( \widetilde{ E}_{\mu}^{[a,a+\varepsilon_2],\rho_U ,\varepsilon_2 }\bigcup \bigcup_{1\leq i\leq N_B} \!\! B_{n_i} \right)\nonumber\\
& {\geq \frac{3}{4}\mu(B)+\frac{3}{4}\mu(B)-\mu(B)=\frac{1}{2}\mu(B).}\nonumber
\end{align*}
 By a slight abuse of notations, up to an extraction, we still denote  by $\left\{B_{n_i}\right\}_{1\leq i\leq N_B}$ the balls $B_{n_i}$ such that $B_{n_i} \cap \widetilde{ E}_{\mu}^{[a,a+\varepsilon_2],\varepsilon_2,\rho_U }\neq \emptyset $. {The last  inequality implies  that the family of balls $\left\{B_{n_i}\right\}_{1\leq i\leq N_B}$ can be chosen so that it verifies conditions $(i_1)$ and $(i_2)$, as well as the two  following additional conditions:}
\begin{align*} 
  {(i_3 )} \ \ \ \ \ & \mu(B_{n_i}) >0 \ \ \mbox{ and } \ \ \mu\Big (\bigcup_{1\leq i\leq  {N_B}}B_{n_i} \Big)\geq \frac{\mu(\widering{B})}{2}=\frac{\mu(B)}{2},\\
  {(i_4)} \ \ \ \ & \mbox{for every $ 1\leq i\leq N_B$},  \ \ \ \ \  {B_{n_i} \cap \widetilde{ E}_{\mu}^{[a,a+\varepsilon_2],\rho_U ,\varepsilon_2}\neq \emptyset.}
 \end{align*} 
 The obtained family is still denoted by $\mathcal{F}^{B}$. 

Applying again Lemma \ref{distcov} to $\mathcal{F}^{B}$ with $v=4$, as in step one (see \eqref{recouetape1}, $(i^{\prime})$ and $(ii^{\prime})$), if one must consider a subfamily, one can assume that the family $\mathcal{F}^{B}$ satisfies $(i_1)$ and $(i_4)$ as well as the following condition $(i_2^{\prime})$ and $(i_3 ^{\prime})$:\mk
\begin{itemize}
\item[$(i_2^{\prime}):$]  for every $ 1\leq i\neq j\leq  {N_B}$, one has $ 4 B_{n_i}\cap 4 B_{n_j}=\emptyset.$ \mk
\item[$(i_3^{\prime}):$]  $\mu(B_{n_i}) >0 \ \ \mbox{ and } \ \ \mu\Big (\bigcup_{1\leq i\leq  {N_B}}B_{n_i} \Big)\geq \frac{\mu(\widering{B})}{2Q_{d,4}}=\frac{\mu(B)}{2Q_{d,4}},$
\end{itemize} 

\sk

Finally  one defines
$$\mathcal{W}_{2}=\bigcup_{U \in\mathcal{W}_1}\bigcup_{B\in\mathcal{F}^{1,U }}\mathcal{F}^{B} \ \ \ \mbox{ and } \ \ \ W_{2}=\bigcup_{L\in\mathcal{W}_{2}}L.$$

The pre-measure $\eta$ is then extended to the $\sigma$-algebra generated by the elements of $\mathcal{W}_{2}$ by setting for every $U \in \mathcal{W}_{1}$, every $B\in \mathcal{G}^{U}$ and   $V \in  \mathcal{F}^{B}$,
\begin{equation}
\label{etaB2ani}
   \eta (V)=\eta (B)\times \frac{\mu(B^{[V]})}{\sum_{V^{\prime} \in \mathcal{F}^{B} }\mu(B^{[V^{\prime}]})} .
 \end{equation}
 
By construction, one has $\sum _{V \in  \mathcal{F}^{B}} \eta(V) = \eta(B)$. Also, \eqref{eq24ani},\eqref{etaB2ani},  and $(i_3 ^{\prime})$ imply
 \begin{equation}
 \label{eqmicani}
 \frac{\eta (V)}{\mu(B^{[V]})} \leq 2Q_{d,4}\frac{\eta (B)}{\mu(B)}\leq \vert V\vert ^{-\varepsilon_2},
 \end{equation}
so that by \eqref{bonnemajo2ani} and  \eqref{eqmicani} one has  
 \begin{align}
 \label{eq25ani}
\eta (V) &\leq 2Q_{d,4}\frac{\eta (B)}{\mu(B)} \times \mu(B^{[V]} )\leq    {|B^{[V]}| ^{-\varepsilon_2}}   |V|^{s_2}    \leq  |V|^{s_2-\ep_2} .
\end{align}

  
\subsubsection{Recurrence scheme and  end of the construction} 

Let $p\in\mathbb{N}^{*}$ be an integer, and set $\mathcal{W}_0=\R^d$. Suppose  that   sets of balls $\mathcal{W}_1$, ..., $\mathcal{W}_{p}$ as well as the measure $\eta $ are constructed such that :
\begin{enumerate}
\mk\item 
 for every $  1\leq q\leq p$, $\mathcal{W}_q \subset  \{U_{n} \}_{n\geq q}$, $\mathcal{W}_q\subset \mathcal{W}_{q-1}$, and $\eta$ is defined on the $\sigma$-algebra generated by the elements of $\bigcup_{q= 1}^p \mathcal{W}_{q}$.

 \mk
\item
For every $ 1\leq q\leq p-1$, for every $U \in\mathcal{W}_{q}$,  setting, as in step 2, $E_U =\limsup_{n\in\mathcal{N}_\mu(\mathcal{B},\mathcal{U},s_q)}B_n \cap U\cap E_{\mu}^{[\alpha,\beta],\varepsilon_q}$, then $\mathcal{H}^{s_q}_{\infty}(E_U )>0$.  If $m_{E_U}^{s_q}$ stands for  the measure associated with $E_U$ provided by Proposition \ref{carl}, there exists $\rho_U >0$ such that, for every $0\leq k\leq\lfloor \frac{\beta-\alpha}{\varepsilon_q}\rfloor+1$, setting $\theta_k=\theta_k^{(q)}=\alpha+k\varepsilon_q$, one has 
$$m_{E_U}^{s_q} (E_U \cap E_{\mu}^{[\theta_k ,\theta_{k+1}],\rho_U ,\varepsilon_q})\geq \frac{1}{2}m_{E_U}^{s_q} (E_U \cap E_{\mu}^{[\theta_k ,\theta_{k+1}],\varepsilon_q}).$$

In particular,
  $$m_{E_U}^{s_q} (E_U \cap \bigcup_{0\leq k\leq\lfloor \frac{\beta-\alpha}{\varepsilon_q}\rfloor+1} E_{\mu}^{[\theta_k ,\theta_{k+1}],\rho_U ,\varepsilon_q})\geq \frac{1}{2}.$$  
\mk
\item 
For every  $  1\leq q\leq p-1$, for every $ U\in\mathcal{W}_{q}$, there exists  a finite family $\mathcal{G}^{U}$  of balls  $B(x,r_x/t_x)$, where $x$, $r_x <\frac{1}{3}\min\left\{\vert \widetilde{U}\vert:\widetilde{U}\in\mathcal{W}_{q}\right\}$ and $t_x$  satisfy  \eqref{recap-ani}, \eqref{recap0ani},  \eqref{recap1ani},  \eqref{recap2ani} and \eqref{eq22ani}. Also, if $B\neq B' \in \mathcal{G}^{U}$, $3B\cap 3B' =\emptyset$. 

Also, for every $B\in \mathcal{G}^{U}$, \eqref{etaSani}  and \eqref{eq23ani} hold true. Moreover $\mathcal{W}_{q+1}\subset \bigcup_{U\in \mathcal{W}_q}  \mathcal{G}^{U}$.

\mk
\item
For every  $  1\leq q\leq p-1$, for every $ U\in\mathcal{W}_{q}$,  for every $B\in \mathcal{G}^{U}$ there exists a family $\mathcal{F}^{B} \subset \left\{U_{n} \right\}_{n\geq q}$  of  pairwise disjoint open  sets such that :
\begin{itemize}

\item for every $\widetilde{U}\neq \widehat{U} \in\mathcal{F}^{B}$, one has
\begin{equation}
\label{disjani}
4B^{[\widetilde{U}]}\cap 4 B^{[\widehat{U}]} =\emptyset;
\end{equation}

\item for every $\widetilde{U}\in \mathcal{F}^{B}$, $\widetilde{U}\subset \widering{B}$, \eqref{etaB2ani}  and \eqref{eq25ani} hold true,   as well as 
\begin{equation}
\label{majornani}
 2Q_{d,4}   \frac{\eta(B)}{\mu(B)} \leq   |B^{[\widetilde{U}]}|^{-\varepsilon_{q+1}}
\end{equation}
 and 
 \begin{equation}
\label{hypsupani}
 { B^{[\widetilde{U}]}\cap \widetilde{ E}_{\mu}^{[\theta_{k_B} ,\theta_{k_B+1}],\rho_{U},\varepsilon_{q+1}}\neq \emptyset;}
\end{equation}

\item  the following inequality also holds true:
\begin{equation}
\label{eq31ani}
\mu\left (\bigcup_{\widetilde{U}\in \mathcal{F}^{B}} B^{[\widetilde{U}]}\right) \geq \frac{\mu(B)}{2Q_{d,4}}.
\end{equation}
\end{itemize} 
\end{enumerate}

\mk

In item (3), the fact that $3B\cap 3B' =\emptyset$ just follows from the choice of $B(x,r_x/t_x)$ instead of simply $B(x,r_x)$.

The proof follows then exactly and rigorously the same lines as those of Step 2. We do not reproduce it here, the only differences are that   $\mathcal{W}_1$, $\mathcal{W}_2$ and $s_2$ are  replaced by $\mathcal{W}_p$, $\mathcal{W}_{p+1}$ and $s_{p+1}$.

\medskip

Finally, define  the Cantor set
$$K=\bigcap_{p\geq 1}W_{p}= \bigcap_{p\geq 1} \bigcup_{V\in \mathcal{W}_p}B^{[V]}.
$$
Applying Caratheodory's extension Theorem to the pre-measure $\eta$ yields a probability outer-measure on $\R^d$ that we still denote by $\eta$, which is metric, so that Borel sets are $\eta$-measurable and its restriction to Borel sets belongs to $\mathcal M(\R^d)$. The so obtained measure  $\eta$ is fully supported on $K$. Also, for every $p\geq 2$, for any $U\in\mathcal{W}_p$, $B\in \mathcal{G}^{U}$, and $\widetilde{U}\in \mathcal{F}^{B}$, the inequalities  \eqref{etaSani}, \eqref{eq23ani},\eqref{etaB2ani} and \eqref{eq25ani} holds with $s_p$ and $\varepsilon_p$ instead of $s_2$ and~$\varepsilon_2$.

\subsubsection{Upper-bound for the mass of a ball}  

One first recall the following lemma (see, e.g.,~\cite{BV}).
\begin{lemme}
\label{geo}
Let $A=B(x ,r)$ and $B=B(x^{\prime} ,r^{\prime})$ be two  closed balls, $q\geq 3$ such that $A\cap B\neq \emptyset$ and $A\setminus (qB)\neq \emptyset$. Then $r^{\prime}\leq r$ and $qB\subset 5A.$
\end{lemme}

Define the gauge function $\zeta:\mathbb{R}^{+}\mapsto \mathbb{R}^{+}$ as follows: 
\begin{itemize}
\item
if   for some $p\geq 1$, $\frac{1}{3}\min\left\{\vert U \vert: \ U\in\mathcal{W}_{p+1}\right\}\leq r< \frac{1}{3}\min\left\{\vert U \vert: \ U\in\mathcal{W}_p\right\}$,  then $\zeta(r)=2Q_{d,4}10^d r^{s_p-5\varepsilon_p }$,\mk 
\item
if $r \geq \frac{1}{3}\min\left\{\vert U \vert: \ U\in\mathcal{W}_{1}\right\}$, $\zeta(r)=1$,\mk
\item
$\zeta(0)=0$.\mk
\end{itemize}
Since $\ep_p\to 0$, one checks that   $\lim_{r\to 0^+}\frac{ \log(\zeta(r))}{\log(r)}=\min\left\{s(\mu,\mathcal{B},\mathcal{U}),\dimm(\mu)\right\}$. 
 
 \medskip
 
Let  $A$ be a ball of radius $r$. If there exists $n\in\mathbb{N}$ such that $A$ does not intersect $K_n$ then $\eta(A)=\eta(A\cap K_n)=0.$ Suppose that for every $n\in\mathbb{N}$, $A$ intersects $K_n$. The goal is to prove that $\eta(A) \leq \zeta(|A|)$ when $|A|$ is small.

\medskip

 Some cases must be distinguished.

First if for every $n\in\mathbb{N}$, $A$ intersects only one contracted set $V_n$ of $K_n$, then by  \eqref{eq23ani}
$$\eta(A)\leq \eta(V_n)\leq  \vert V_n \vert^{s_n-\varepsilon_n }\underset{n\rightarrow+\infty}{\rightarrow}0.$$

In the other case, there exists $p \in\mathbb{N}$ such that $A$ intersects only one element of $\mathcal{W}_{p}$,  and at least two elements of $\mathcal{W}_{p +1}$. Denote by $U$ the unique  element of $\mathcal{W}_p$  intersecting~$A$.

\begin{enumerate}
\item {\bf Case 1:} If $ |A| \geq |U|  $, then by \eqref{eq25ani}
\begin{equation}
\label{res0ani}
\eta(A)\leq \eta(U)\leq |U| ^{s_p-\varepsilon_p}\leq \zeta(  |A|).
\end{equation}
\item {\bf Case 2:} If   $|A| < |U|  $ and $A$ intersects at least two  balls of $\mathcal{G}^{U}$: Observe that when $A$ intersects two balls  $B$ and $B'$ of $\mathcal{G}^{U}$,  since by item $(3)$ of the recurrence scheme $3B\cap 3B' =\emptyset$, one necessarily has (by Lemma \ref{geo}) $B\cup B' \subset 5A$. Hence,    $\bigcup_{B\in \mathcal{G}^{U}:  B\cap A\neq \emptyset}B \subset 5A$ and by \eqref{etaSani}  and \eqref{eq22ani},
\begin{align*}
\eta(A) & =  \eta(U )\times \frac{\sum_{B\in \mathcal{G}^{U}:  B\cap A\neq \emptyset}   m_{E_{U}}^{s_{p+1}}(B)}{{\sum_{B' \in \mathcal{G}^{U}}m_{E_U}^{s_{p+1}}(B')}} \leq \frac{4 Q_{d,1}} {C_{\ep_{p+1},d}}    \eta(U)    m_{E_{U}}^{s_{p+1}}(5A )   .
\end{align*}
Then, by \eqref{majcarl}, \eqref{eq21ani}, \eqref{eq24ani} and \eqref{eqmicani}
\begin{align}
\label{res1ani}
\eta(A) &\leq  \frac{4 Q_{d,1}} {C_{\ep_{p+1},d}}  \eta(U )   \kappa_d \frac{ (5|A|)^{s_{p+1}}}{ \mathcal{H}^{\mu,s_{p+1}}_{\infty}(E_U )}  
\leq    5^{s_{p+1}} \frac{4 Q_{d,1}\kappa_d} {C_{\ep_{p+1},d}}      \frac{ \eta(U )  } {  \mu(B^{[U]}) }   { |A|^{s_{p+1}}}\nonumber\\  
&\leq |A|^{s_{p+1}} |U|^{-2\ep_{p}} \leq |A|^{s_{p+1}-2\ep_{p+1}}\leq    \zeta(|A|),
\end{align}
where we used that$ |A|< |U|$, and the mappings  $x\mapsto \vert U\vert ^{-x}$ and $x\mapsto x^{-\ep_{p+1}}$ are decreasing.

\mk
\item {\bf Case 3:} If  $A$ intersects only one ball of $\mathcal{G}^{U}$:  calling  $B$ this particular ball and $r_B$ its radius (at this stage there should be no confusion with  the radii of the terms of the sequence $(B_n)_{n\ge 1}$), two cases must again be distinguished: 
\begin{enumerate}
\mk
\item {\bf Subcase 3.1:  $|B|\leq |A|$:}  by  \eqref{eq23ani},
 \begin{align}
 \label{res2ani}
\eta(A) & \leq \eta(B)  \leq |B|^{s_{p+1}-\ep_{p+1}}     \leq  |A|^{s_{p+1}-\ep_{p+1} } \leq\zeta(|A|).
\end{align}

\mk
\item  {\bf Subcase 3.2:  $|A|\leq |B|$:}
Denote by $k_B$ the integer such that its center belongs to $E_{\mu}^{[\theta_{k_B} ,\theta_{k_B+1}] ,\rho_{U},\varepsilon_{p+1}}$.

The ball $A$ must intersect  at least two elements $V\neq V'$ of $\mathcal{W}_{p+1}$ (by definition of $p$). Note that those sets must belong to $\mathcal{F}^{B}$ (because $A$ intersects only $B$).  Applying Lemma \ref{geo} to the ball $A$ with any of those ball $V\in\mathcal{F}_{p+1}$ , since $A\cap V \neq \emptyset$ and $A\setminus B^{[V]}\neq \emptyset$ (because $A$ intersects an other dilated ball, $B^{[V^{\prime}]}$ by hypothesis and two such balls verifies  \eqref{disjani}), one has 
\begin{equation}
\label{incani}
\bigcup_{V\cap A\neq \emptyset}B^{[V]} \subset 5A.
\end{equation}

Then, \eqref{etaB2ani} and \eqref{eq31ani} imply that
 \begin{align}
 \label{majoxcani}
\eta(A) & =  \eta (B)\cdot \frac{ \sum_{V\in\mathcal{W}_{p +1}: V\cap A\neq \emptyset} \mu(B^{[V]})}{\sum_{V' \in \mathcal{F}^{B} }\mu(B^{[V^{\prime}]})} \leq 2Q_{d,4}\frac{\eta (B)}{\mu(B)}  \mu(5A).
\end{align}

Recalling \eqref{incani}, the ball $5A$ contains some of the balls of $\mathcal{F}^{B}$: Hence, by \eqref{hypsupani},  $\widetilde{ E}_{\mu}^{[\theta_{k_B} ,\theta_{k_B+1}],\rho_{U},\varepsilon_{p+1}}\cap 5A \neq \emptyset$. Since $\vert A\vert \leq \vert B\vert$,   by \eqref{recap-ani}, since $r_B< \frac{1}{10}\rho_{U}$, for any  $x\in \widetilde{ E}_{\mu}^{[\theta_{k_B} ,\theta_{k_B+1}],\rho_{U},\varepsilon_{p+1}}\cap 5A$ one has

\begin{equation}
\label{majo3cani}
\mu(5A)\leq \mu(B(x,10r))\leq (10r)^{\theta_{k_B}-2\varepsilon_{p+1}}.
\end{equation}
Recalling  \eqref{recap1bisani} (applied to the ball $B$), one has 
\begin{equation}
\label{majozcani}
 \mu(B)\geq (r_B)^{\theta_{k_B}+2\varepsilon_{p+1}}.
 \end{equation} 

 Using \eqref{eq23ani} (applied to $B$) \eqref{majoxcani}, \eqref{majo3cani} and \eqref{majozcani}, one obtains
 \begin{align*}
 \eta(A) &\leq 2Q_{d,4}r_B^{s_{p+1}-\varepsilon_{p+1}} \frac{\big(10 r\big)^{\theta_{k_B}-2\varepsilon_{p+1}}}{r_B^{\theta_{k_B}+2\varepsilon_{p+1}}} \\
 &=2Q_{d,4} 10^{\theta_{k_B}-2\varepsilon_{p+1}} \frac{r_B^{s_{p+1}-\theta_{k_B} -\frac{\varepsilon_{p+1}}{\delta}-2\varepsilon_{p+1}}}{r^{s_{p+1}-\theta_{k_B}-\varepsilon_{p+1}-2\varepsilon_{p+1}}} r^{s_{p+1}-\theta_{k_B}-\varepsilon_{p+1}-4\varepsilon_{p+1}}\\
 &\leq 2Q_{d,4} 10^{\theta_{k_B}-2\varepsilon_{p+1}} r^{s_{p+1}-5\varepsilon_{p+1}}.
  \end{align*}
 Finally, recalling \eqref{defskani},  $s_{p+1}  -5\varepsilon_{p+1} \leq \alpha\leq \theta_k $, and since $r_B\geq r$ and $s_p \leq s_{p+1}$,  one gets 
 $$
 \eta(A)\leq 2 Q_{d,4} 10^{\theta_k-\varepsilon_{p+1}}\big(r \big)^{s_{p+1}-5\varepsilon_{p+1}}\leq 2Q_{d,4}  10^{d-\varepsilon_{p+1}}\vert A\vert ^{s_{p}-5\varepsilon_{p}}\leq  2Q_{d,4}  10^{d-\varepsilon_{p+1}}\vert A\vert ^{s_{p}-5\varepsilon_{p}},
$$ 
hence
\begin{equation}
 \label{res3ani}\eta(A) \leq  \zeta(|A|).
 \end{equation}
\end{enumerate}
 \end{enumerate} 
Since for any $p\in\mathbb{N}$ and any ball $A$ satisfying $\vert A \vert \leq \frac{1}{3}\min\left\{\vert U \vert:U\in\mathcal{W}_p\right\}$, if $A$ intersects at most one element of $\mathcal{W}_p$, the inequalities \eqref{res0ani}, \eqref{res1ani}, \eqref{res2ani}, \eqref{res3ani} proves that for any such ball, one has
 $\eta(A)\leq   \zeta(|A|)$.

 Hence recalling Definition \ref{hausgau}, by the mass distribution principle,  one deduces that $\mathcal{H}^{\zeta}(K)\geq 1 $, which concludes the proof of Theorem \ref{lowani}.

\subsection{Proof of Corollary~\ref{minoeffec}}
\label{sec-mino}

\subsubsection{Some basic properties about the $\mu$-essential Hausdorff content}
{In this sub-section,  basic properties of the $\mu$-essential content are established.}

\mk

First, we work in this article with the $\vert \vert \cdot \vert \vert_{\infty}$ norm for convenience. Any other norm could have been chosen, the corresponding quantities would have been equivalent.

\mk

In  \eqref{hcont}, only closed balls are considered. Choosing open balls does not change the value of \eqref{eqmucont} in Definition \ref{mucont}.

{The following propositions are directly derived from the properties of the standard Hausdorff measures. 


\begin{proposition}
\label{propmuc}
 {Let $\mu \in\mathcal{M}(\R^d)$, $s\geq 0$ and $A\subset \R^d $ be a Borel set. The $s$-dimensional $\mathcal{H}^{\mu,s}_{\infty}(\cdot)$ outer measure satisfies the following properties:}
\begin{enumerate}
\item  { If $\vert A\vert\leq 1$, the mapping $s\geq 0\mapsto \mathcal{H}^{\mu,s}_{\infty}(A)$ is decreasing from $\mathcal{H}^{\mu,0}_{\infty}(A)=1$ to $\lim_{t\to +\infty}\mathcal{H}^{\mu,t}_{\infty}(A)=0$.\mk
\item    $0\leq \mathcal{H}^{\mu,s}_{\infty}(A)\leq \min\left\{\vert A\vert ^s, \mathcal{H}^{s}_{\infty}(A)\right\}$.

\mk
\item For every subset $ B\subset A$ with $\mu(A)=\mu(B)$,   $\mathcal{H}^{\mu,s}_{\infty}(A)=\mathcal{H}^{\mu,s}_{\infty}(B).$}

\mk
\item For every  $\delta \geq 1$,  $\mathcal{H}^{\mu,\frac{s}{\delta}}_{\infty}(A)\geq (\mathcal{H}^{\mu,s}_{\infty}(A))^{\frac{1}{\delta}}.$

\mk
\item  For every  $  s>\overline{\dim}_H (\mu)$,   $\mathcal{H}^{\mu,s}_{\infty}(A)=0.$\mk

%
%
\end{enumerate}
\end{proposition}
\begin{proof}
{Items $(1)$, $(2)$, $(3)$ directly follow  from the  definition. Item $(4)$ is   obtained by concavity of the mapping $x\mapsto |x|^{1/\delta}$.} 

\mk

{(5)} By Definition \ref{dim}, for any $s>\overline{\dim}_H (\mu)$, there exists a set $E$ with $\dim_{H}(E)<s$ and $\mu(E)=1.$ Using item (2), one has then   $0\leq\mathcal{H}^{\mu,s}_{\infty}(A)=\mathcal{H}^{\mu,s}_{\infty}(A\cap E)\leq \mathcal{H}^{s}_{\infty}(A\cap E)\leq \mathcal{H}^{s}(E)=0$.
\sk

\end{proof}
\subsubsection{Proof of Corollary \ref{minoeffec}}
 One starts with a lemma, the proof of which can be found in \cite{ED2}.

\begin{lemme} 
\label{gscainf}
Let   $\mu  \in\mathcal{M}(\R^d)$.
Let    $\mathcal{B} =(B_n :=B(x_n ,r_n))_{n\in\mathbb{N}}$ be a $\mu $-a.c sequence of balls of $ \R^d$.
 Then  for every $\varepsilon>0$, there exists  a  $\mu $-a.c sub-sequence $(B_{\phi(n)})_{n\in\mathbb{N}}$ of  $\mathcal{B} $ such that  for every $n\in\mathbb{N}$, $\mu (B_{\phi(n)})\leq (r_{\phi(n)})^{\underline{\dim}_H (\mu )-\varepsilon}.$
\end{lemme}
 

\begin{proof}[Proof of  Corollary~\ref{minoeffec}](1) Observe that item  (2) of Proposition \ref{propmuc} implies that $ t(\mu,\delta,\varepsilon ,\mathcal{B}) \geq \dimm(\mu)-\ep$, and  $ t(\mu,\delta ,\mathcal{B}) \geq \dimm(\mu)$.


Now choose  $\varepsilon>0$ so small  that $\frac{(1-\varepsilon)\dimm(\mu)}{\delta\cdot  t(\mu, \delta,\varepsilon ,\mathcal{B} )}\leq 1$. Recalling Lemma \ref{gscainf}, up to an extraction, one can assume that for any $n\in\mathbb{N}$, $$\mu(B_n)\leq \vert B_n \vert ^{(1-\varepsilon^2).\dimm(\mu)}.$$ 

Due to   \eqref{tdeltainf}, there exists $N_{\varepsilon}\in\mathbb{N}$ such that for any $n\geq N_{\varepsilon}$, $$\mathcal{H}^{\mu ,\dimm(\mu)-\varepsilon}_{\infty}(\widering B_n ^{\delta})\geq \vert B_n ^{\delta} \vert ^{(1+\varepsilon). t(\mu,\delta,\varepsilon ,\mathcal{B}) }.
$$ 
Then, Proposition \ref{propmuc} (4) implies that for every $n\geq N_{\varepsilon}$,  
\begin{align*}
\mathcal{H}_{\infty}^{\mu, \frac{(1-\varepsilon)\dimm(\mu)\times(\dimm(\mu)-\varepsilon)}{\delta \cdot t(\mu,\delta,\varepsilon ,\mathcal{B})  }}(\widering B_n ^{\delta})&\geq(\mathcal{H}^{\mu,\dimm(\mu)-\varepsilon}_{\infty}(B_n ^{\delta}))^{\frac{(1-\varepsilon)\dimm(\mu)}{\delta\cdot  t(\mu,\delta,\varepsilon ,\mathcal{B}) }}\\
&\geq \vert B_n ^{\delta} \vert ^{\frac{ (1+\varepsilon)  t(\mu,\delta,\varepsilon ,\mathcal{B}) }{\delta \cdot t(\mu,\delta,\varepsilon ,\mathcal{B})}(1-\varepsilon).\dimm(\mu)}\\
&\geq \vert B_n  \vert ^{ (1+\varepsilon)(1-\ep)  \dimm(\mu)} 
\geq \mu(B_n) .
\end{align*}
Thus, setting $s_{\delta,\varepsilon} =  \frac{(1-\varepsilon)\dimm(\mu)\times(\dimm(\mu)-\varepsilon)}{\delta \cdot t(\mu,\delta,\varepsilon ,\mathcal{B}) }$, Corollary~\ref{zzani} yields
$$\dim_{H}(\limsup_{n\rightarrow+\infty}\widering B_n ^{\delta})\geq s_{\delta,\varepsilon}.$$
Since the result holds for any $\varepsilon>0$, one gets the desired conclusion. 
\end{proof}

\section{Estimation of essential content for self-similar measures}
\label{sec-parti}
In this section one computes  the Hausdorff content of balls in the case of the Lebesgue measure, and estimates it for any self-similar measure.

\subsection{Computation of essential content for the Lebesgue measure}

When the measure $\mu$ is the Lebesgue measure, the computations are quite easy.
 

\begin{proposition}
Let $B=B(x,r) $ be a ball in $\R^d$,  and $\mathcal{L}^d$ be the $d-$dimensional Lebesgue measure. Then for any $0\leq s\leq d$, 
$\mathcal{H}^{\mathcal{L}^d,s }_{\infty}(B)=\mathcal{H}^{\mathcal{L}^d ,s}_{\infty}(\widering{B})=r^s .$
\end{proposition}


\begin{proof}
One starts first by computing $\mathcal{H}^{\mathcal{L}^d ,d}_{\infty}(B)$.

Let  $\ep>0$, and let $E\subset B$ be a Borel set with $\mathcal{L}^d (E)=\mathcal{L}^d (B)$. Notice first that since $B$ covers $E$, recalling that $\mathbb{R}^d$ is endowed with  $\vert \vert \cdot  \vert \vert_{\infty}$ one has $$\mathcal{H}^{\mathcal{L}^d,d}_{\infty}(E)\leq \mathcal{H}^{d}_{\infty}(B)\leq \vert B \vert^d.$$ 

Consider   a sequence of balls  $(L_n)_{n\in\mathbb{N}}$   such that $$\mathcal{H}^{d}_{\infty}(E)\leq\sum_{n\geq 0}\vert L_n \vert ^d \leq (1+\varepsilon)\mathcal{H}^{d}_{\infty}(E).$$
This implies 
\begin{align*}
(1+\varepsilon)\vert B\vert ^d &\geq (1+\varepsilon)\mathcal{H}^{d}_{\infty}(B )\geq (1+\varepsilon)\mathcal{H}^{d}_{\infty}(E)\geq\sum_{n\geq 0}\vert L_n \vert^d \\
& \geq \sum_{n\geq 0}\mathcal{L}^d (L_n)\geq \mathcal{L}^d (E)=\mathcal{L}^d (B)=\vert B\vert ^d.
\end{align*}
Taking the infimum on the Borel sets $E \subset B$ such that $\mathcal{L}^d (E)=\mathcal{L}^d (A)$ gives
$$  \vert B\vert ^d \leq (1+\varepsilon)\mathcal{H}^{\mathcal{L}^d,d }_{\infty}(B) .$$ 
In particular, 
$$ \frac{1}{1+\varepsilon}\vert B\vert ^d \leq \mathcal{H}^{\mathcal{L}^d ,d }_{\infty}(B)\leq \vert B\vert ^d .$$ Letting $\varepsilon \to 0$ shows that $\mathcal{H}^{\mathcal{L}^d,d }_{\infty}(B)= \vert B\vert ^d$.
This implies, with item $(4)$ of Proposition \ref{propmuc}, that for any $\delta \geq 1$, 
$$\vert B\vert ^{\frac{d}{\delta}}\geq \mathcal{H}^{ \mathcal{L}^d,\frac{d}{\delta}}_{\infty}(B)\geq (\mathcal{H}^{\mathcal{L}^d,d}_{\infty}(B) )^{\frac{1}{\delta}}=\vert B\vert ^{\frac{d}{\delta}},$$
hence the result.
\end{proof}


\subsection{Proof of Theorem \ref{contss}}


\begin{proposition} 
\label{autosim2}
Let $\mu$ be a self-similar measure.   For any $0<\varepsilon\leq \dim(\mu)$, there exists a constant $\kappa( d,\mu,\varepsilon)\in(0,1)$ such that for any ball $B=B(x,r)$ with $x\in K$ (the attractor of the underlying IFS) and $r\leq 1$,   one has $$\kappa( d,\mu,\varepsilon)\vert B\vert ^{\dim(\mu)-\varepsilon}\leq\mathcal{H}^{\mu, \dim(\mu)-\varepsilon }_{\infty}(\widering{B})\leq  \mathcal{H}^{ \mu, \dim(\mu)-\varepsilon}_{\infty}(B)\leq\vert B\vert ^{\dim(\mu)-\varepsilon}.$$
In addition, for any $s>\dim( \mu)$ one has  $\mathcal{H}^{\mu,s}_{\infty}(B)=0.$
\end{proposition}


\begin{proof}

Let $\left\{f_1 ,...,f_m\right\}$ the underlying IFS. Denote   by $c_i$ the contraction ration of $f_i$,  and $(p_1 ,...,p_m)$ the probability vector with positive entries associated with $\mu$ so that \eqref{def-ssmu2} is satisfied. 
Set $\alpha=\dim(\mu)$ and  $\Lambda=\left\{1,...,m\right\}$. 
For  $k\geq 0$ and $\underline{i}:=(i_1 ,...,i_k)\in\Lambda^k$, define  
\begin{itemize}
\item $c_{\underline{i}}=c_{i_1}...c_{i_k}$, $f_{\underline{i}}=f_{i_1}\circ ... \circ f_{i_k}$ and $K_{\underline{i}}=f_{\underline{i}}(K),$ so that  $\vert K_{\underline{i}}\vert = c_{\underline{i}}\vert K\vert$. \medskip
\item $\Lambda ^{(k)}=\left\{\underline{i}:=(i_i ,..., i_s) \ : \ c_{i_s}2^{-k}< c_{\underline{i}}\leq 2^{-k}  \right\}$.
\end{itemize}

Note first that item (5) of Proposition \ref{propmuc} implies that for any $s>\dim(\mu)$,  $\mathcal{H}^{\mu,s}_{\infty}(B)=0.$

Let us consider $0\leq s<\dim_H (\mu)$ and start by few remarks.

 Recalling \eqref{emu} and Proposition  \ref{mle}, let us fix $\rho_{\varepsilon}$ so that $\mu(E_{\mu} ^{[\alpha,\alpha], \rho_{\varepsilon},\varepsilon})\geq \frac{1}{2}$ and write $E=E_{\mu}^{[\alpha,\alpha] ,\rho_{\varepsilon}, \varepsilon}.$ 

Set $\Lambda ^{*}:=\bigcup_{k\geq 0}\Lambda ^k $, and for $\underline{i}\in\Lambda^*$, define $E_{\underline{i}}=f_{\underline{i}}(E)$ and $\mu_{\underline{i}}=\mu(f_{\underline{i}}^{-1})$. One has
\begin{align}
\nonumber E_{\underline{i}}
\nonumber &=\left\{f_{\underline{i}}(x)\in \R^d : \, x\in K \forall\, r\leq \rho_{\varepsilon}, \  \mu(B(x,r))\leq r^{\alpha -\varepsilon}\right\}\\
\nonumber &=\left\{f_{\underline{i}}(x):\, x\in K, \, \forall, c_{\underline{i}} r\leq c_{\underline{i}} \rho_{\varepsilon}, \  \mu(f_{\underline{i}} ^{-1}(B(f_{\underline{i}}(x),r c_{\underline{i}})))\leq \left(\frac{r c_{\underline{i}}}{c_{\underline{i}}}\right)^{\alpha -\varepsilon}\right\}\\
\label{eq77} &=\left\{y\in K_{\underline{i}} : \, \forall r'\leq c_{\underline{i}}\rho_{\varepsilon} , \  \mu_{\underline{i}}(B(y,r'))\leq \left(\frac{r'}{c_{\underline{i}}}\right)^{\alpha -\varepsilon}\right\},
\end{align}
 Also,  
$\mu_{\underline{i}}(E_{\underline{i}})=\mu(E)\geq \frac{1}{2}.$

One emphasizes that iterating the self-similarity equation gives 
$$\mu=\sum_{\underline{i}^{\prime}\in\Lambda ^k}p_{\underline{i}^{\prime}}\mu_{\underline{i}^{\prime}}, $$
which implies that $\mu_{\underline{i}}$ is absolutely continuous with respect to $\mu$ (since all  $p_{\underline{i}}$'s are strictly positive).  

\sk

We are now ready to estimate the $\mu$-essential content of a ball $B$ centered on $K$.

\sk

Let $B=B(x,r)$ with $x\in K$ and $r\leq \min_{1\leq i\leq m}c_i$.

\sk

Since $x\in K$, there exists an $\underline{i }$ such that $\min_{1\leq j\leq m}c_j r\leq c_{\underline{i}}\vert K\vert\leq r$ and $K_{\underline{i}}\subset \widering{B}$.

%
%
%
%
%
%
%
By construction, $E_{\underline{i}} \subset \widering{B}$.
\sk

Consider a Borel set $A\subset B$ such that $\mu(A)=\mu(B).$ One aims at giving a lower-bound of the Hausdorff content of $A$ which does not depends on $A$. 

\sk

 Consider a sequence of balls  $(L_n=B( z_n,\ell_n))_{n\geq 1}$ covering $A\cap E_{\underline{i}} $, such that  $\ell_{n}<\rho_{\varepsilon} c_{\underline{i}}$ and $ z_n\in A\cap E_{\underline{i}} $.  
Since $\mu_{\underline{i}}$ is absolutely continuous with respect to $\mu$, it holds that $\mu_{\underline{i}}(A)=1.$ 

By \eqref{eq77} applied to every $n\in\mathbb{N}$ , one has $\left(\frac{\vert L_n \vert}{c_{\underline{i}}}\right)^{\alpha-\varepsilon}\geq \mu_{\underline{i}}(L_n)$, so that
\begin{equation}
\label{eq78}
\sum_{n\in\mathbb{N}}| L_n| ^{\alpha-\varepsilon}\  \geq \sum_{n\in\mathbb{N}}c_{\underline{i}} ^{\alpha -\varepsilon}\mu_{\underline{i}}(L _n )\geq c_{\underline{i}} ^{\alpha -\varepsilon}\mu_{\underline{i}}\left(\bigcup_{n\in \N }L_n \right)\geq c_{\underline{i}} ^{\alpha -\varepsilon}\mu_{\underline{i}}(E_{\underline{i}})\geq \frac{1}{2}c_{\underline{i}} ^{\alpha -\varepsilon}.
\end{equation}
This series of inequalities holds for any sequence of balls $(L_n)_{n\in\mathbb{N}}$ with radius less than $\rho_{\varepsilon} c_{\underline{i}}$  centered on $A\cap E_{\underline{i}}$ and covering $A\cap E_{\underline{i}}$.

Now, assume that $(L_n)_{n\in\mathbb{N}}$ is a sequence of balls  covering $A\cap E_{\underline{i}} $, which still verifies  $\ell_{n}<\rho_{\varepsilon} c_{\underline{i}}$ but  $z_n$ does not necessarily belongs to $A\cap E_{\underline{i}} $. 

Let $n\in\mathbb{N}$. One constructs recursively a sequence of balls $(L_{n,j})_{1\leq j\leq J_n}$ \ such that the following properties hold for any $1\leq j\leq J_n$:

\begin{itemize}
\item[•]  $L_{n,j}$ is centered on $A\cap E_{\underline{i}}\cap L_n$;\mk
\item[•] $A\cap E_{\underline{i}}\cap L_n \subset \bigcup_{1\leq j \leq J_n}L_{n,j}$;\mk
\item[•] for all $1\leq j\leq J_n,$ $\vert L_{n,j}\vert= \vert L_n \vert$;\mk
\item[•]the center of $L_{n,j}$ does not belong to any $L_{n,j^{\prime}}$ for $1\leq j^{\prime}\neq j\leq J_n$. \mk
\end{itemize}

To achieve this, simply consider $y_1\in A\cap E_{\underline{i}}\cap L_n$ and set $L_{1,n}=B(y_1,\ell_n).$ If $A\cap E_{\underline{i}}\cap L_n \nsubseteq L_{1,n}$, consider $y_2 \in A\cap E_{\underline{i}}\cap L_n \setminus L_{1,n}$ and set $L_{2,n}=B(y_2 ,\ell_n)$. If $A\cap E_{\underline{i}}\cap L_n \nsubseteq L_{1,n}\cup L_{2,n}$, consider $y_3 \in A\cap E_{\underline{i}}\cap L_n \setminus  L_{1,n}\cup L_{2,n}$ and set $L_{3,n}=B(y_3 ,\ell_n)$, and so on...

 Note that, for any $1\leq j\leq J_n$, any ball $L_{j,n}$ has radius $\ell_n$, intersects $L_n$ (which also has radius $\ell_n$) and, because $y_j \notin \bigcup_{1 \leq j^{\prime}\neq j\leq J_n}L_{j^{\prime},n}$, it holds that, for any $j\neq j^{\prime}$, $\frac{1}{3}L_{n,j}\cap \frac{1}{3}L_{n,j^{\prime}}=\emptyset$. By Lemma \ref{dimconst}, this implies that $J_n \leq Q_{d, \frac{1}{3}}.$

Hence, denoting by $(\widetilde L_n)_{n\in \N} $ the collection of the corresponding balls centered on $A\cap E_{\underline{i}} $ associated with all the balls $L_n$, one has by \eqref{eq78} applied to $(\widetilde L_n)_{n\in\mathbb N}$:
$$\sum_{n\in\mathbb{N}}|L_n| ^{\alpha-\varepsilon}\geq \frac{1}{Q_{d,\frac{1}{3}}} \sum_{n\in\mathbb{N}}|\widetilde L_n| ^{\alpha-\varepsilon} \geq \frac{1}{2Q_{d,\frac{1}{3}}} c_{\underline{i}} ^{\alpha -\varepsilon}.$$
Remark also that any ball of radius smaller that $r$ can be covered by at most $(\frac{2}{\rho_{\varepsilon}})^d$ balls of radius $r\rho_{\varepsilon}$.  

This proves that, for any sequence of balls $\widehat{L}_n$ covering $A\cap E_{\underline{i}}$, since $c_{\underline{i}}\geq \frac{\min_{1\leq j\leq m}c_j }{\vert K\vert }\vert B\vert$, it holds that 

\begin{equation}
\label{loceq}
\sum_{n\in\mathbb{N}}|\widehat{L}_n| ^{\alpha-\varepsilon} \geq \frac{\rho_{\varepsilon}^d}{2^{d+1}Q_{d,\frac{1}{3}}} c_{\underline{i}} ^{\alpha -\varepsilon}\geq \frac{\min_{1\leq j\leq m}c_j ^{\alpha-\varepsilon}\rho_{\varepsilon}^d}{\vert K\vert^{\alpha-\varepsilon} 2^{d+1}Q_{d,\frac{1}{3}}} \vert B\vert ^{\alpha -\varepsilon}.
\end{equation} 
Recall \eqref{hcont}. Since \eqref{loceq} is valid for any covering $(\widehat{L}_n )_{n\in\mathbb{N}}$ of $A\cap E_{\underline{i}}$, one has
\begin{equation}
\vert B\vert^{\alpha-\varepsilon}\geq \mathcal{H}^{\alpha-\varepsilon}_{\infty}(A)\geq \mathcal{H}^{\alpha-\varepsilon}_{\infty}(A \cap E_{\underline{i}})\geq \frac{\min_{1\leq j\leq m}c_j ^{\alpha-\varepsilon}\rho_{\varepsilon}^d}{\vert K\vert^{\alpha-\varepsilon}2^{d+1}Q_{d,\frac{1}{3}}} \vert B\vert ^{\alpha -\varepsilon}.
\end{equation}  
Taking the infimum over all the Borel sets $A\subset B$ satisfying $\mu(A)=\mu(B)$, one gets
$$\vert B\vert^{\alpha-\varepsilon}\geq \mathcal{H}^{\mu,\alpha-\varepsilon}_{\infty}(B)\geq \frac{\min_{1\leq j\leq m}c_j ^{\alpha-\varepsilon}\rho_{\varepsilon}^d}{\vert K\vert^{\alpha-\varepsilon} 2^{d+1}Q_{d,\frac{1}{3}}} \vert B\vert ^{\alpha -\varepsilon}. $$
The results stands for balls of diameter less than $\min_{1\leq j \leq m}c_j $. Then for any ball $B$ centered on $K$ with $\vert B\vert\leq 1$, remarking that
$$\vert B\vert^{\alpha-\varepsilon}\geq \mathcal{H}^{\mu,\alpha-\varepsilon}_{\infty}(B)\geq \mathcal{H}^{\mu,\alpha-\varepsilon}_{\infty}(\min_{1\leq j\leq m}c_j B)$$ 
and setting $\kappa(d, \mu, \varepsilon)= \frac{\min_{1\leq j\leq m}c_j ^{2(\alpha-\varepsilon)}\rho_{\varepsilon}^d}{\vert K\vert^{\alpha-\varepsilon} 2^{d+1}Q_{d,\frac{1}{3}}}$ yields the desired inequality.
\end{proof}

\begin{remark}
Note that in the proof of Proposition \ref{autosim2}, the estimate of $\mathcal{H}^{\mu,s}_{\infty}(B)$ for $s<\dim (\mu)$ only relies on the absolute continuity of $\mu(f_{\underline{i}}^{-1}(\cdot))$, for any $\underline{i}\in\Lambda^*$. In particular, the same estimates holds for any quasi-Bernoulli measures (which are proved to be exact-dimensional, see \cite{H}). 
\end{remark}

This result in hand, one establishes the more general Theorem \ref{contss}.

 \begin{proof}[Proof of Theorem \ref{contss}]
Note first, that by item $(5)$ of Proposition \ref{propmuc}, for any $s>\dim(\mu)$ and any set $E$, one has $ \mathcal{H}^{\mu,s}_{\infty}(E)=0.$ 
 \sk
 
Let us fix $s<\dim (\mu)$ and set $\varepsilon=\dim(\mu)-s >0.$ Since $K\cap \Omega \subset \Omega$ and $\mu(K\cap \Omega)=\mu(\Omega)$, it holds that 
$$\mathcal{H}^{\mu,s}_{\infty}(\Omega)\leq \mathcal{H}^{s}_{\infty}(\Omega \cap K).$$ 
 
 It remains to show that there exists a constant $c( d,\mu,s)$ such that for any open set $\Omega$, the converse inequality
 $$c(d,\mu,s)\mathcal{H}^{s}_{\infty}(\Omega \cap K)\leq \mathcal{H}^{\mu,s}_{\infty}(\Omega)$$
holds.
 \sk
 
 Let  $E \subset \Omega$ be a Borel set such that $\mu(E)=\mu(\Omega)$ and  
 \begin{equation}
 \label{equat0}
 \mathcal{H}^{s}_{\infty}(E)\leq 2\mathcal{H}^{\mu,s}_{\infty}(\Omega).
 \end{equation}
  Let $\left\{L_n \right\}_{n\in\mathbb{N}}$ be a covering of $E$ by balls verifying
\begin{equation}
\label{equat1}
\mathcal{H}^{s}_{\infty}(E)\leq\sum_{n\geq 0}\vert L_n \vert^s \leq 2\mathcal{H}^{s}_{\infty}(E).
\end{equation}
 The covering $(L_n)_{n\in\mathbb{N}}$ will be modified to get  a covering $(\widetilde{L}_n)_{n\in\mathbb{N}}$ which verifies the following properties:

$\bullet$ $K\cap \Omega\subset \bigcup_{n\in\mathbb{N}}\widetilde{L}_n$;

\mk
$\bullet$ $\bigcup_{n\in\mathbb{N}}L_n \subset \bigcup_{n\in\mathbb{N}}\widetilde{L}_n$; 

\mk

$\bullet $ one has  $$\sum_{n\geq 0}\vert \widetilde{L}_n \vert^s \leq 8\cdot 2^s\frac{Q_{d,1}^2}{\kappa( d,\mu,\varepsilon)}\sum_{n\geq 0}\vert L_n \vert ^s,
$$
where $\kappa( d,\mu,\varepsilon)$ is the constant introduced in Proposition \ref{autosim2} and $Q_{d,1}$ is the constant arising in Proposition \ref{densibesi} applied with $v=1.$ Last item together with  \eqref{equat0} and \eqref{equat1} then immediately imply that
  $$\frac{\kappa( d,\mu,\varepsilon)}{8\cdot 2^s Q_{d,1}^2}\mathcal{H}^{s}_{\infty}(K\cap \Omega)\leq \mathcal{H}^{\mu,s}_{\infty}(\Omega),$$
and setting $c( d,\mu,s)=\frac{\kappa( d,\mu,\dim(\mu)-s)}{ 8\cdot2^s Q_{d,1}^2 }$ concludes the proof.

\mk

Let us start the construction of the sequence of balls $(\widetilde{L}_n)_{n\in\mathbb{N}}$.  Let  $X=( K \setminus \bigcup_{n\in\mathbb{N}}L_n )\cap \Omega$. For every $x\in X$, fix $ 0< r_x \leq 1$ such that $B(x,r_x)\subset \Omega$.   One of the following alternatives must occur:
\medskip
\begin{enumerate}
\item for any ball $L_n$ such that $L_n \cap B(x,r_x)\neq \emptyset$, it holds that $\vert L_n \vert \leq r_x$, or \mk 
\item there exists $n_x \in\mathbb{N}$ such that $L_{n_x}\cap B(x,r_x)\neq \emptyset$ and $\vert L_{n_x}\vert \geq r_x$.
\end{enumerate} 

Consider the set $S_1$ be the set of points in $X$ for which the first alternative holds. By Lemma~\ref{besimodi} applied with $v=1$, it is possible to extract from the covering of $S_1$, $\left\{B(x,r_x),x\in S_1\right\}$, $Q_{d,1}$ families of pairwise disjoint balls, $\mathcal{F}_1 ,...,\mathcal{F}_{Q_{d,1}}$, such that
$$S_1 \subset \bigcup_{1\leq i\leq Q_{d,1}}\bigcup_{L\in\mathcal{F}_i}L.$$
 Now, any ball $L_n$ intersecting a ball $L\in \bigcup_{1\leq i\leq Q_{d,1}}\mathcal{F}_i$ must satisfy $\vert L_n \vert \leq |L|.$ In particular, since for any $1\leq i\leq Q_{d,1}$ the balls of $\mathcal{F}_i$ are pairwise disjoint, applying Lemma \ref{dimconst} to the balls of $\mathcal{F}_i$ intersecting $L_n$ we get that $L_n$ intersects at most $Q_{d,1}$ balls of $\mathcal{F}_i$, hence at most $Q_{d,1}^2$ balls of  $\bigcup_{1\leq i\leq Q_{d,1}}\mathcal{F}_i$. 

\sk
 
Let $L \in \bigcup_{1\leq i\leq Q_{d,1}}\mathcal{F}_i$. One aims at replacing all the balls $L_n$ intersecting $L$ by the ball $2L$.

%
 
 For any $1\leq i \leq Q_{d,1}$ and any ball $L\in\mathcal{F}_i$, denote by $\mathcal{G}_L$ the set of balls $L_n$ intersecting $L$. Since $E\subset\bigcup_{n\in\mathbb N} L_n$ and $\mu(E)=\mu(\Omega)$, one has  $E\cap L\subset \bigcup_{B\in \mathcal{G}_L}B$ and $\mu(E\cap L)=\mu(L)$. By Definition \ref{mucont} and Proposition \ref{autosim2},  this implies that
 \begin{equation}
 \label{equat2}
\kappa( d,\mu,\varepsilon)\vert L\vert^s \leq\mathcal{H}^{\mu,s}_{\infty}(L)\leq \sum_{B\in\mathcal{G}_L}\mathcal{H}^{\mu,s}_{\infty}(B)\leq \sum_{B\in\mathcal{G}_L}\vert B\vert^s.
 \end{equation}
 Replace the balls of $\mathcal{G}_L$ by the ball $\widehat L=2L$ (recall that $\bigcup_{B\in\mathcal{G}_{L}}B \subset 2L$). The new sequence of balls so obtained by the previous construction applied to all the balls $L\in\bigcup_{\leq i\leq Q_{d,1}}\mathcal{F}_i$ is denoted by $(\widehat{L}_k)_{1\le k \le K}$, where $0\le K\le+\infty$. 

It follows from the construction and \eqref{equat2} that $S_1 \subset \bigcup_{1\le k\le K}\widehat{L}_k$ and 
\begin{equation}
\label{equat3}
\sum_{1\le k\le K}\Big (\frac{\vert \widehat{L}_k \vert}{2}\Big)^s \leq \frac{Q_{d,1}^2}{\kappa(d,\mu,\varepsilon)}\sum_{n\geq 0}\vert L_n \vert^s.
\end{equation}

On the other hand, since for any $x\in S_2 = X\setminus S_1$, there exists $n_x\in\mathbb{N}$ such that ${L}_{n_x}\cap B(x,r_x)\neq \emptyset$ and $r_x \leq \vert {L}_{n_x}\vert$, one has $S_2\subset \bigcup_{n\in\mathbb{N}}2{L}_n$, so that 
$$
 \Big (\bigcup_{n\in\mathbb N}L_n\Big )\cup  \Big (K\cap \Omega \setminus \bigcup_{n\in\mathbb N}L_n\Big ) \subset \Big (\bigcup_{1\le k\le K}\widehat L_k\Big )\cup \Big ( \bigcup_{n\in\mathbb N}2{L}_n\Big ) .
 $$
Putting the elements of $(\widehat{L}_k)_{1\le k \le K}$ and $(2L_n)_{n\ge 0}$ in a single sequence $(\widehat{ L}_n)_{n\ge 0}$, writing $(\widetilde{L}_n :=2\widehat{L}_n)_{n\in\mathbb{N}}$, by construction, $K\cap \Omega \subset \bigcup_{n\in\mathbb N}\widetilde{L}_n$ and due to \eqref{equat3}:
$$\mathcal{H}^ {s}_{\infty}(K\cap \Omega)\leq\sum_{n\in\mathbb N}\vert \widetilde{L}_n \vert^s \leq 2^s\Big (\frac{Q_{d,1}^2}{\kappa(d,\mu,\varepsilon)}+1\Big )\sum_{n\in\mathbb N}\vert L_n \vert^s \leq 8\cdot2^s \frac{Q_{d,1}^2}{\kappa( d,\mu,\varepsilon)}\mathcal{H}^{\mu,s}_{\infty}(\Omega) .$$
%
 \end{proof}
\begin{remark}
The proof of Theorem \ref{contss} only uses Proposition \ref{prop-ss}. In particular, Theorem \ref{contss} holds for any measure $\mu \in \mathcal{M}(\mathbb{R}^d)$ supported on $K$ and verifying, for any $\underline{i}\in\Lambda^{*}$, $\mu(f_{\underline{i}}^{-1}(\cdot))$ is absolutely continuous with respect to $\mu$. 
\end{remark}

 \section{Applications of Theorem \ref{lowani}}
 \label{sec-example}
 \subsection{Ubiquity Theorems for self-similar measures}







 \subsubsection{Proof of Theorem \ref{prop-ss}}


\mk

Let $\mu$ be a self-similar measure with support $K$, and 
set $\alpha=\dim(\mu).$ Let $(B_{n}:=B(x_n,r_n))_{n\in\mathbb{N}}$ be a sequence of balls such that  $x_n\in K$ for all $n\in\mathbb N$, $\lim_{n\to+\infty} r_{n} = 0$  and $\mu(\limsup_{n\rightarrow+\infty}B_n)=1.$  

Fix  $\varepsilon>0$, $v>1$ and $\delta \geq 1$ and set $\mathcal{B}_v =\left\{vB_n\right\}_{n\in\mathbb{N}}$. 
Lemma~\ref{equiac} shows that $\mathcal{B}_v$ is $\mu$-a.c. Then, by Proposition \ref{autosim2}, for $n$ large enough, one  has 
\begin{align*}
 \mathcal{H}_{\infty}^{\mu,\alpha-\varepsilon}(\widering{(v B_{n} ) ^{\delta}})\ge\kappa(d,\mu,\varepsilon)(vr_{n} )^{\delta(\alpha-\varepsilon)}\geq (v r_{n} )^{\delta(\alpha-\frac{\varepsilon}{2})}.
\end{align*}
Consequently, $$t(\mu,\delta,\varepsilon,\mathcal B_v)=\limsup_{n\rightarrow+\infty}\frac{\log \mathcal{H}_\infty^{\mu,\alpha-\varepsilon}(\widering{(v B_{n} )^{\delta}})}{\delta\log \vert v B_{n} \vert}\leq \alpha-\frac{\varepsilon}{2}$$
so $t(\mu,\delta,\varepsilon,\mathcal B_v)\leq \alpha.$
Due to Corollary~\ref{minoeffec}, one concludes that
$$ \dim_{H}(\limsup_{n\rightarrow+\infty}(vB_n) ^{\delta})\geq \frac{\alpha}{\delta}.$$ 
But for any $\ep'>0$, $\limsup_{n\rightarrow+\infty}(v B_n)^{\delta}\subset\limsup_{n\rightarrow+\infty}B_n ^{\delta -\ep'}$, so that 
$$\dim_{H}(\limsup_{n\rightarrow+\infty}B_n ^{\delta -\ep'})\geq \frac{\alpha}{\delta}.$$
It follows that for any $\ep'>0$ and $\delta\geq 1$ one has 
$$ \dim_{H}(\limsup_{n\rightarrow+\infty}B_n ^{\delta})\geq \frac{\alpha}{\delta+\ep'}.$$
Letting $\ep'\to 0$ proves that $\dim_{H}(\limsup_{n\rightarrow+\infty}B_n ^{\delta})\geq \frac{\dim(\mu)}{\delta}$, hence the result.

\begin{remarque}
\label{remarkballopti}
If the sequence of balls $(B_n)_{n\in\mathbb{N}}$ is not assumed to be $\mu$-a.c, but only to verify $\mu(\limsup_{n\rightarrow+\infty}B_n)=1$, then the same lower-bound estimate  holds for $\dim_H (\limsup_{n\to\infty} B_n^\delta)$, but the existence of a gauge function as in Theorem \ref{lowani} does not hold in general.
\item[•] Let us also notice that the computation in the proof of Theorem \ref{prop-ss} actually shows that, under the assumption that $\lim_{n\to+\infty}\frac{\log \mu(B_n)}{\log \vert B_n \vert}=\dim (\mu)$, it holds that, for $n$ large enough, $\mathcal{H}^{\mu,s}_{\infty}(B_n ^{\delta})\geq \mu(B_n) \Leftrightarrow s<\frac{\dim (\mu)}{\delta}.$
\end{remarque}

 \subsubsection{Proof of Theorem \ref{rectss}}

Given $ \tau_1 =1\leq \tau_2  \leq ...\leq \tau_{d}$ and $s\ge 0$,  set $\tau=(\tau_1,\dots,\tau_{d})$ and 
$$
g_{\tau }(s)=\max_{1\leq k\leq d}\left\{s\tau_{k}-\sum_{1\leq i\leq k}\tau_k-\tau_i\right\}.
$$

We will need the following lemma (one refers to \cite{KR}, Proposition 2.1 for the proof, although it is stated in terms of singular values functions).

 \begin{lemme}
 \label{rectc}
 Let $ \tau_1 =1\leq \tau_2  \leq ...\leq \tau_{d}$. 

 The are two positive constants $C_1$ and $C_2$ depending on $d$ only such that for all $s\ge 0$, $r>0$ and $x\in\mathbb{R}^d$ one has 
 $$
 C_1 r^{g_{\tau}(s)}\leq\mathcal{H}^{s}_{\infty}(R_{\tau}(x,r))=\mathcal{H}^{s}_{\infty}(\widering R_{\tau}(x,r))\leq C_2 r^{g_{\tau}(s)}.
 $$
 \end{lemme}

Recall that $K$ is the closure of its interior, and note that since the weights $p_i $ are taken positive in Definition \ref{def-ssmu},  one must have $\mu(\widering{K})>0.$ 

Denote $\widetilde{\mu}=\frac{\mu(\cdot)}{\mu(\widering{K})}$ and $\alpha =\dim (\mu)=\dim (\widetilde{\mu}).$  It is easily verified that the computation made in the proof of Theorem  \ref{contss} implies that, for any open set $\Omega \subset \widering{K}$, there exists a constant $c(d,\mu,s)$ given by Theorem \ref{contss}, so that
\begin{equation}
\label{contmut}
\begin{cases} c(\mu,d,s)\mathcal{H}^s _{\infty}(\Omega)\leq \mathcal{H}^{\tilde{\mu},s}_{\infty}(\Omega)\leq \mathcal{H}^{s}_{\infty}(\Omega)\text{ if } s<\alpha\\
\mathcal{H}^{\tilde{\mu},s}_{\infty}(\Omega)=0\text{ if } s>\alpha.
\end{cases}
\end{equation}

Also, $\widetilde{\mu}$ being absolutely continuous with respect to $\mu$, the sequence $(B_n)_{n\in\mathbb{N}}$ is $\widetilde{\mu}$-a.c. Furthermore, up to a $\widetilde{\mu}$-a.c extraction, we can assume that each ball $(B_n)_{n\in \mathbb{N}}$ is included in $\widering{K}$ (and we will do so). 


 Let $\varepsilon>0$. Set $\mathcal{R}=\left\{R_n\right\}_{n\geq 0}.$
 By Lemma \ref{gscainf}, up to a $\widetilde{\mu}$-a.c extraction, one can assume that for every $n\in\mathbb{N}$, the ball $B_n$ satisfies 
 $$\widetilde{\mu}(B_n)\leq r_n ^{\alpha-\varepsilon}.$$
 Setting $\tau^{\prime}=(\frac{\tau_i}{\tau_1})_{1\leq i\leq d}$, for all $0\le s<\alpha-\varepsilon$, one has $$g_{\tau^{\prime}}(s)=\max_{1\leq k\leq d}\left\{\frac{s\tau_k -\sum_{1\leq i\leq k}\tau_k-\tau_i}{\tau_1}\right\}.$$
  From equation  \eqref{contmut} and Lemma \ref{rectc}, one deduces that
 \begin{equation}
 C_1 c(d,\mu,s) r_n^{\tau_1 g_{\tau^{\prime}}(s)}\leq \mathcal{H}^{\tilde{\mu},s}_{\infty}(R_n).
 \end{equation}
 In particular, for any $s$ verifying 
\begin{equation}
\label{equab} 
 \tau_1 g_{\tau^{\prime}}(s)\leq \alpha-\frac{\varepsilon}{2},
\end{equation} 
 if $r_n\le 1$ one has
 $$
 C_1 c(d,\mu,s) r_n ^{\alpha-\frac{1}{2}\varepsilon}\leq C_1 c(d,\mu,s) r_n^{ \tau_1 g_{\tau^{\prime}}(s)}\leq \mathcal{H}^{\tilde{\mu},s}_{\infty}(R_n).
 $$
Since $r_n \to 0$, for $n$ large enough,  this yields

 \begin{equation}
 \label{equaa}
 \widetilde{\mu}(B_n)\leq r_n ^{\alpha-\varepsilon}\leq C_1 c(d,\mu,s) r_n^{\tau_1 g_{\tau^{\prime}}(s)}\leq \mathcal{H}^{\tilde{\mu},s}_{\infty}(R_n),
 \end{equation}
hence $(B_n)_{n\in\mathcal{N}_{\tilde\mu}(\mathcal{B},\mathcal{R},s)}$ is  $\widetilde \mu$-a.c., and  $s(\widetilde{\mu}, \mathcal{R},\mathcal{B})\geq   s$.

It remains to note that 

\begin{align}
\label{shtrou}
 \eqref{equab}&\Leftrightarrow \max_{1\leq k\leq d}\left\{\frac{s\tau_k -\sum_{1\leq i\leq k}\tau_k-\tau_i}{\tau_1}\right\}\leq \frac{\alpha-\frac{\varepsilon}{2}}{\tau_1}\nonumber \\
& \Leftrightarrow \forall 1\leq k\leq d, \ \frac{s\tau_k -\sum_{1\leq i\leq k}\tau_k-\tau_i}{\tau_1}\leq  \frac{\alpha-\frac{\varepsilon}{2}}{\tau_1}\nonumber\\
&\Leftrightarrow \forall 1\leq k\leq d, \ s\leq \frac{\alpha-\frac{1}{2}\varepsilon+\sum_{1\leq i\leq k}\tau_k -\tau_i}{\tau_k}\nonumber\\
& \Leftrightarrow s\leq \min_{1\leq k\leq d}\left\{\frac{\alpha-\frac{1}{2}\varepsilon+\sum_{1\leq i\leq k}\tau_k -\tau_i}{\tau_k}\right\}.
\end{align}

Since $\varepsilon>0$ was arbitrary, this implies that $$s(\widetilde{\mu}, \mathcal{R},\mathcal{B})\geq  \min_{1\leq k\leq d}\left\{\frac{\alpha+\sum_{1\leq i\leq k}\tau_k -\tau_i}{\tau_k}\right\},$$ 
and applying Theorem \ref{lowani} gives the desired lower bound estimate.
 
\begin{remarque}
\label{remarkrect}
Note that the estimates made in the proof of Theorem \ref{rectss}, together with Lemma~\ref{rectc}, can be used to show that, under the assumption that $\lim_{n\to\infty}\frac{\log \mu(B_n )}{\log \vert B_n \vert}=\dim (\mu)$, one has the following properties: 

If $s <\min_{1\leq k\leq d}\left\{\frac{\dim (\mu)+\sum_{1\leq i\leq k}\tau_k -\tau_i}{\tau_k}\right\}$ then,  for $n\in\mathbb{N}$ large enough, $\mathcal{H}^{\mu,s}_{\infty}(R_n)\geq \mu (B_n)$. If $s>\min_{1\leq k\leq d}\left\{\frac{\dim (\mu)+\sum_{1\leq i\leq k}\tau_k -\tau_i}{\tau_k}\right\}$ then, for $n$ large enough, $\mathcal{H}^{\mu,s}_{\infty}(R_n)\leq \mu (B_n).$  
\end{remarque}

\subsection{Application to self-similar shrinking targets}

\begin{proof}

In this section, Theorem \ref{shrtag} is proved and we adopt the notation of the proof of Proposition \ref{autosim2}.

Set $s=\dim_H (K)$. Note that, for each $k\in\mathbb{N}$, the set $\left\{B(f_{\underline{i}}(x),2\vert K\vert c_{\underline{i}})\right\}_{\underline{i}\in \Lambda^{k}}$ covers $K.$ In particular, for any measure $\mu$ supported on  $K$, the family $\left\{B(f_{\underline{i}}(x),2\vert K\vert c_{\underline{i}})\right\}_{\underline{i}\in \Sigma^*}$ is  $\mu$-a.c. 

Set $B_{\underline{i}}=B(f_{\underline{i}}(x),2\vert K\vert c_{\underline{i}}).$ One now focuses on proving that, for any $\delta\geq 1$, $\dim_H (\limsup_{\underline{i}\in\Sigma^*}B_{\underline{i}}^{\delta})=\frac{s}{\delta}.$ If this holds, since for any $\varepsilon>0$,
 \begin{equation}
 \label{enca1}
 \limsup_{\underline{i}\in\Sigma^{*}}B_{\underline{i}}^{\delta+\varepsilon}\subset\limsup_{\underline{i}\in \Sigma^*}B(f_{\underline{i}}(x), c_{\underline{i}}^{\delta})\subset \limsup_{\underline{i}\in\Sigma^{*}}B_{\underline{i}}^{\delta},
  \end{equation}
it also holds that $\dim_H (\limsup_{\underline{i}\in \Sigma^*}B(f_{\underline{i}}(x), c_{\underline{i}}^{\delta}))=\frac{s}{\delta}.$

Note that $s$ satisfies the equation $\sum_{1\leq i\leq m}c_i ^s =1.$ Let also be $\nu_s$, the  measure on $(\Sigma, \mathcal{B}(\Sigma))$ associated with the probability vector $(p_i =c_i ^s )_{1\leq i\leq m}$ and $\mu_s$ its projection on $K$ by the canonical coding map. 

Let $0<t<\min_{1\leq i\leq m}c_i$ and 
$$\Lambda ^{\left(k\right)}_t=\left\{\underline{i}=\left(i_1 ,\ldots, i_{\ell}\right) \in\Lambda^* : \ c_{i_{\ell}}t^k< c_{\underline{i}}\leq t^k  \right\}.
$$
 If $\underline{i}\in \Lambda^{(k)}_t$, then for any $\ell \in \Lambda$, the word $\underline{i}\ell \notin \Lambda^{(k)}_t .$ This implies that for any $\underline{i}\neq \underline{j}\in \Lambda^{(k)}_t$, $[\underline{i}]\cap [\underline{j}]=\emptyset.$

Then for any $\delta \geq 1$, one has, for any $\varepsilon>0$, 
\begin{align*}
\sum_{k\geq 0}\sum_{\underline{i}\in\Lambda^{\left(k\right)}_t} \left(\vert B_{\underline{i}}\vert^{\delta}\right)^{\frac{s+\varepsilon}{\delta}}&=\left(4\vert K\vert\right)^{s+\varepsilon}\sum_{k\geq 0}\sum_{\underline{i}\in\Lambda^{\left(k\right)}_t} c_{\underline{i}}^{s+\varepsilon}\\
&\leq \left(4\vert K\vert\right)^{s+\varepsilon}\sum_{k\geq 0}\sum_{\underline{i}\in\Lambda^{\left(k\right)}_t}t^{k\varepsilon}\nu_s\left([\underline{i}]\right).
\end{align*}   
Since $\sum_{\underline{i}\in\Lambda^{\left(k\right)}_t}\nu_s\left(\underline{i}\right)\leq 1$, one obtains
\begin{align*}
\sum_{k\geq 0}\sum_{\underline{i}\in\Lambda^{\left(k\right)}_t} \left(\vert B_{\underline{i}}\vert^{\delta}\right)^{\frac{s+\varepsilon}{\delta}}\leq \left(4\vert K\vert\right)^{s+\varepsilon}\sum_{k\geq 0}t^{k\varepsilon}<+\infty.
\end{align*}
This shows that 
  \begin{equation}
  \label{majoshri}
  \displaystyle \dim_H \Big(\limsup_{\underline{i}\in \Lambda^*}B_{\underline{i}}^{\delta}\Big)\leq \frac{s}{\delta}.
  \end{equation}

One now establishes the lower-bound estimate. By the dimension regularity assumption (see Definition \ref{dimreg}), $\dim_H (\mu_s)=s.$ Since $(B_{\underline{i}})_{\underline{i}\in\Sigma^*}$ is $\mu_s$-a.c,  Theorem~\ref{prop-ss} yields $\dim_H (\limsup_{\underline{i}\in\Sigma^{*}}B_{\underline{i}}^{\delta})\geq \frac{s}{\delta}$.

\end{proof}

\subsection{Study of a problem related to a question of Mahler}

Let us first notice that by Theorem \ref{approxrat}, one has $$\dim_{H}\Big (\limsup_{B\in\mathcal Q}B^\delta\cap K_{1/3}^{(0)}\Big )\leq \min\left\{\frac{1}{\delta},\frac{\log 2}{\log 3}\right\}.$$
In particular, this proves that the expected upper-bound in Theorem \ref{main} stands.

\mk 
 
Before showing that the lower-bound also holds, let us start with some facts and remarks.

\mk

\begin{remarque}
$\bullet$  One has $\mathcal{H}^{\frac{\log 2}{\log 3}}_{\infty}(K_{1/3})>0$ (this is well known and easily follows from the fact that $K_{1/3}$ carries an Alfhors regular measure of dimension $\frac{\log 2}{\log 3}$). 

Moreover, for every $k\in\mathbb{N}$, setting $\mathcal{K}_k =\left\{ f_{\underline{i}}([0,1]) \right\}_{\underline{i}\in\Lambda^k}$, one has 
\begin{equation}
\label{eqegal1}
1=\sum_{I \in\mathcal{K}_k}\vert I\vert^{\frac{\log 2}{\log 3}}.
\end{equation}

\mk

$\bullet$ For every $k\in\mathbb{N},$ let us define 
\begin{equation}
\label{defomega}
\Omega_{k}=\bigcup_{I\in\mathcal{K}_k}\overset{\circ}{I}.
\end{equation}

Since $\mathcal{H}^{\frac{\log2}{\log 3}}_{\infty}(\bigcup_{I\in\mathcal{K}_k}I \setminus \Omega_k )=0$ (it is a finite set of points), it follows from \eqref{eqegal1} that 
\begin{equation}
\label{contom}
C\mathcal{H}^{\frac{\log 2}{\log 3}}_{\infty}(\Omega_{k})\leq\mathcal{H}^{\frac{\log 2}{\log 3}}_{\infty}(K_{1/3})\leq \mathcal{H}^{\frac{\log 2}{\log 3}}_{\infty}\Big (\bigcup_{I\in\mathcal{K}_k}I\Big )=\mathcal{H}^{\frac{\log 2}{\log 3}}_{\infty}(\Omega_{k})\le 1,
\end{equation} 
with  $C=\mathcal{H}^{\frac{\log 2}{ \log 3}}_{\infty}(K_{1/3})>0$. 
\mk

$\bullet$ If $n\in\mathbb{N}$ and $T\in\mathcal{T}_n =\left\{[\frac{k}{3^{n}},\frac{k+1}{3^n}[,0\leq k\leq 3^n-1 \right\}$ is a triadic interval of generation $n$, denote by $F_T$ the canonical homothetical mapping which sends $[0,1]$ to $ \overline{T}.$  For every  $I\in\bigcup_{J\in\mathcal{K}_k}F_T (J)$, for all  $n\leq k^{\prime}\leq n+k$ and all $x=(x_n)_{n\in\mathbb{N}}\in\Sigma$ such that $\pi (x)\in I$, one has 
\begin{equation}
\label{equafre1}
S_{n+k^{\prime}}\phi(x)=S_{n}\phi (x)\times \frac{n}{n+k^{\prime}}.
\end{equation}  

\end{remarque}
 \mk
 
 One are now ready to finish the proof of Theorem \ref{main}.

 \mk
 
Let $(\varepsilon_{q} )_{q\in\mathbb{N}}$ be a positive sequence such that $\lim_{q\to \infty}\varepsilon_{q}= 0.$ One constructs  a family $\{U_{p,q,\delta}\}_{\delta\ge 1,\, q\in\mathbb{N},\, 0\leq p\leq q}$ of open sets as follows: Let $\delta\geq 1$, $q\in\mathbb{N}^{*}$ and $0\leq p\leq q$. Consider $T$ a triadic interval of  generation $n_q=\lfloor \log_{3} (q^{2\delta})\rfloor +1$ included in $B(\frac{p}{q},q^{-2\delta}).$ Let $N_{p,q,\delta}$ be large enough to ensure that for any $x\in\Sigma$ verifying $\pi(x)\in T$, one has
\begin{equation}
\label{equafre2}
S_{n_q}\phi(x)\times \frac{n_q}{n_q+N_{p,q,\delta}}\leq \varepsilon_{q}.
\end{equation}
Set 
\begin{equation}
\label{defu}
U_{p,q,\delta}=F_{T}(\Omega_{N_{p,q,\delta}}).
\end{equation}

By \eqref{equafre1} and \eqref{equafre2}, for all $x\in U_{p,q, \delta}$ one has 
$$
S_{n_q+N_{p,q}}\phi(x)\leq \varepsilon_{q}.
$$ 
This implies that   $\bigcap_{Q\ge 1}\bigcup_{q\ge Q}\bigcup_{0\le   p\le q} U_{p,q, \delta}\subset K_{1/3}^{(0)}\cap \bigcap_{Q\ge 1}\bigcup_{q\ge Q}\bigcup_{0\le   p\le q}B(\frac{p}{q},q^{-2\delta})$. 
Since $U_{p,q,\delta}$ is an homothetic copy of $\Omega_{N_{p,q,\delta}}$ (see \eqref{defu}), by \eqref{contom}, due to the choice of $n_q$ there exists   $\widetilde{C}>0$ independent of $p$, $q$ and $\delta$ such that 
\begin{equation}
\label{hcontc}
\mathcal{H}^{\frac{\log 2}{\log 3}}_{\infty}(U_{p,q, \delta})\geq\widetilde{C} q^{-2\delta \frac{\log(2)}{\log (3)}}.
\end{equation} 
For $1\leq \delta\leq \frac{\log 3}{\log 2}$, it follows that 
\begin{equation}
\label{hcont1}
 \mathcal{H}^{\frac{\log 2}{\log 3}}_{\infty}(U_{p,q, \delta})\geq \widetilde{C} q^{-2}=\mathcal{L} \Big(B\Big(\frac{p}{q},q^{-2}\Big)\Big). 
\end{equation}  
For $\delta \geq \frac{\log 3}{\log 2}$,  by concavity of $x\mapsto x^{\frac{\log 3}{\delta \log 2}},$

\begin{align}
\label{hcont2}
\mathcal{H}^{\frac{1}{\delta}}_{\infty}(U_{p,q, \delta})\geq (\mathcal{H}^{\frac{\log 2}{\log 3}}_{\infty}(U_{p,q, \delta}))^{\frac{\log 3}{\delta \log 2}}\geq \widetilde{C}  (q^{-2\delta \frac{\log 2}{\log 3}})^{ \frac{\log 3}{\delta\log 2}}=\widetilde{C}\mathcal{L} \Big(B\Big(\frac{p}{q},q^{-2}\Big)\Big).
\end{align}

By Theorem \ref{lowani} (or by Rams-Koivusalo's Theorem \ref{lowkr}) applied to $\mathcal{Q}=(B(\frac{p}{q},\frac{1}{q^2}))_{q\in\mathbb{N}^* ,0\leq p\leq q}$, $\mathcal{U}=(U_{p,q,\delta})_{q\in\mathbb{N}^* ,0\leq p\leq q}$ and the Lebesgue measure, one gets
\begin{align*}
&\dim_H \Big (\bigcap_{Q\ge 1}\bigcup_{q\ge Q}\bigcup_{0\le   p\le q} U_{p,q, \delta}\Big )\geq \frac{\log 2}{\log 3} &\text{ if }&1\leq \delta \leq \frac{\log 3}{\log 2}\\
&\dim_H \Big (\bigcap_{Q\ge 1}\bigcup_{q\ge Q}\bigcup_{0\le   p\le q} U_{p,q, \delta}\Big )\geq \frac{1}{\delta}&\text{ if  }&\delta\geq \frac{\log 3}{\log 2}.
\end{align*} 

\bibliographystyle{plain}
\bibliography{bibliogenubi}

\end{document}